\documentclass[11pt]{article}
\usepackage{color,xcolor}
\usepackage{graphicx,mathrsfs}
\usepackage{amssymb,bbm}
\usepackage{amsthm}
\usepackage{amsmath}
\usepackage[margin=1in]{geometry}
\usepackage{caption,numbysec}
\usepackage{subcaption}
\usepackage{float}
\usepackage{mdframed,pdfsync,mathrsfs} 
\usepackage{esint}
\usepackage{dsfont}
\usepackage{authblk}
\usepackage[normalem]{ulem}

\newtheorem{assumption}{Assumption}
\usepackage[colorlinks=true,pdfpagemode=UseNone,urlcolor=blue,linkcolor=blue,citecolor=blue,breaklinks=true]{hyperref}

\newcommand{\norm}[1]{\lVert#1\rVert}

\usepackage{enumitem}

\newtheorem{thm}{Theorem}[section]
\newtheorem{lemma}[theorem]{Lemma}

\newtheorem{remark}[theorem]{Remark}

\newtheorem{claim}[theorem]{Claim}

\newcommand{\bea}{\begin{eqnarray*}}
\newcommand{\eea}{\end{eqnarray*}}
\newcommand{\ben}{\begin{eqnarray}}
\newcommand{\een}{\end{eqnarray}}
\newcommand{\beq}{\begin{equation}}
\newcommand{\eeq}{\end{equation}}

\newcommand{\R}{\ensuremath{\mathbb{R}}}

\newcommand{\Rm}{{\mathbb R}}

\newcommand{\E}{\mathbb{E}}

\newcommand{\bal}{\begin{aligned}}
\newcommand{\enbal}{\end{aligned}}
\newcommand{\be}{\begin{equation}}
\newcommand{\ee}{\end{equation}}

\renewcommand{\hat}[1]{\widehat{#1}}

\renewcommand{\d}{\partial}

\newcommand{\frc}{\text{rc}}
\newcommand{\fsc}{\text{sc}}

\begin{document}
\title{Convergence of two-timescale gradient descent ascent dynamics: finite-dimensional and mean-field perspectives}
\author{Jing An} 
\author{Jianfeng Lu}
\affil{Duke University}
\date{}
\maketitle
\numberbysection

\begin{abstract}
The two-timescale gradient descent-ascent (GDA) is a canonical gradient algorithm designed to find Nash equilibria in min-max games. We analyze the two-timescale GDA by investigating the effects of learning rate ratio on convergence behavior in both finite-dimensional and mean-field settings. In particular, for finite-dimensional quadratic min-max games, we obtain long-time convergence in near quasi-static regimes through the hypocoercivity method. For mean-field GDA dynamics, we investigate convergence under a finite-scale ratio using a mixed synchronous-reflection coupling technique.
\end{abstract}

\section{Introduction}
Multi-objective optimization problems, which aim to optimize multiple objectives simultaneously, have gained significant attention due to their wide range of applications, including economics and finance \cite{tapia2007applications}, engineering design \cite{andersson2000survey}, and control systems \cite{gambier2007multi}. In this paper, we focus on a specific bi-objective problem, which is the classical min-max optimization
\begin{equation}\label{objective}
    \min_{x\in\mathcal{X}} \max_{y\in\mathcal{Y}} K(x,y).
\end{equation}
Here $K$ can be considered as the game objective for min and max players, and $x, y$ are player strategies from the strategy space $\mathcal{X}, \mathcal{Y}$.  

Min-max optimization problems are important as they arise naturally when competing objectives or adversarial dynamics are involved in machine learning tasks. Such tasks include generative adversarial networks \cite{goodfellow2014generative}, multi-agent reinforcement learning \cite{busoniu2008comprehensive}, distance computation in optimal transport \cite{lin2020projection}. The goal of min-max optimization problem (\ref{objective}) is to find the pure Nash equilibria \cite{nash1951non}, or saddle points. However, finding the Nash equilibrium (NE) is noticeably subtle, as it may not even exist.

In terms of learning algorithms to solve min-max optimization problems (\ref{objective}), the most standard first-order method is the gradient descent-ascent (GDA) algorithm. At each iteration, the GDA performs gradient descent over $x$ and gradient ascent over $y$. With equal stepsizes for descent and ascent updates, there has been massive literature establishing both asymptotic and non-asymptotic convergence results of the average iterates \cite{korpelevich1976extragradient, nemirovski2004prox, nedic2009subgradient}. For finite-dimensional min-max problems, under the assumption that  $K$ is convex-concave, GDA can achieve linear convergence \cite{cherukuri2017saddle, adolphs2019local, liang2019interaction}. 
However, in more general settings where $K$ lacks convex-concave structure, GDA with equal stepsizes may converge to limit cycles or even diverge \cite{balduzzi2018mechanics, hommes2012multiple, mertikopoulos2018cycles, daskalakis2018training, bailey2020finite}.

To tackle the difficulties caused by nonconvexity-nonconcavity, the
two-timescale GDA algorithm, which takes different learning rates
for gradient descent and ascent updates, has
become a popular alternative. As \cite{heusel2017gans} shows, the
two-timescale GDA empirically has better convergence
performance. Regarding convergence results of the two-timescale GDA in
finite dimensional Euclidean spaces, we mention that the best
last-iterate convergence results that we are aware of assume $f$ only
satisfies weak convexity-concavity, such as two-sided
Polyak-\L{}ojasiewicz (PL) condition \cite{doan2022convergence,
  yang2020global}, under the assumption that a saddle point exists.

Beyond finite-dimensional min-max problems, there has been significant recent progress on min-max problems in the mean-field regime \cite{hsieh2019finding, domingo2020mean, ma2022provably, lu2023two, wang2022exponentially, kimsymmetric, lascu2024fisher, conger2024coupled, cai2024convergence}. These problems, also known as zero-sum games in the space of probability distributions, can be formulated as
\begin{equation}\label{eq:Kxy}
    \min_{p\in\mathcal{P}(\mathcal{X})} \max_{q\in\mathcal{P}(\mathcal{Y})} \int_{\mathcal{Y}} \int_{\mathcal{X}} K(x,y) p(x) q(y) \, dx dy,
\end{equation}
Compared with the finite dimensional case, the pure strategies $x, y$ are replaced by mixed strategies $p,q$ which are probability distributions over the set of strategies.
The solution $(p^*,q^*)$ to such infinite-dimensional min-max problems is called mixed Nash equilibrium (MNE). Compared to the existence of NE for pure strategies, existence theory of MNE is better established. For example, by Glicksberg's minimax theorem for continuous games \cite{glicksberg1952further}, such a MNE exists if $\mathcal{X}, \mathcal{Y}$ are compact manifolds
without boundary.  However, finding the MNE is in general difficult. Recent works such as \cite{domingo2020mean} suggest an infinite-dimensional entropy-regularized min-max problem, so that a unique Nash equilibrium is guaranteed to exist and has an explicit form. The mean-field counterpart of GDA, known as Wasserstein gradient descent-ascent flows (or mean-field GDA), has become a popular topic of study, particularly regarding its convergence guarantees and rates \cite{wang2024open}. 

In terms of whether the two-timescale mean-field GDA with a fixed finite scale ratio $\eta>0$ converges to the unique MNE \cite{domingo2020mean, wang2024open}, previous established results are unfortunately limited: For compact manifold domains, \cite{ma2022provably} obtained convergence guarantees in the quasi-static regime (where $\eta = +\infty$ or $0$). This was later extended by \cite{lu2023two} through a novel Lyapunov construction under log-Sobolev inequality assumptions, but still the convergence is limited near the quasi-static regime where $\eta\gg 1$ or $\eta\ll 1$. For the whole Euclidean spaces, a recent work \cite{cai2024convergence}  provides exponential convergence guarantees for mean-field GDA with $\eta=1$, assuming that the objective function $K$ is smooth and globally strongly convex-concave.

In this work, we investigate the two-timescale GDA in the continuous-time limit, and conduct convergence analysis from new perspectives in partial differential equations and stochastic analysis. In particular, if we take $\eta=\eta_y/\eta_x$ as the ratio of learning rates, we want to answer the following question:
\begin{center}
    \textit{How does the two-timescale GDA converge depending on the ratio of learning rates $\eta$?}
\end{center}
We will provide quantitative answers for both finite-dimensional and infinite-dimensional cases.
\subsection{Our contributions}
We highlight the major contributions as follows.
\begin{itemize}
    \item In the finite-dimensional setting, we analyze the dynamics of the two-timescale GDA algorithm applied to the classical quadratic game \cite{letcher2019differentiable}. Through rescaling, we identify that the learning dynamics exhibit a hypocoercive structure \cite{villani2009hypocoercivity}. Leveraging this framework, we construct a Lyapunov function to quantitatively estimate the convergence rate with respect to ratio of learning rates $\eta$ (Theorem \ref{thm:hypocoercivity}). For completeness, we further discuss how to design a preconditioner for the two-timescale dynamics to reduce sensitivity to extreme values of $\eta$.
    
    \item In the infinite-dimensional / mean-field setting, we examine the entropy-regularized min-max problem introduced by \cite{domingo2020mean} and study the convergence of the two-timescale mean-field GDA for a finite range of $\eta$ using the mixed synchronous-reflection coupling method \cite{eberle2016reflection}. Our convergence analysis accommodates \textit{locally nonconvex-nonconcave} objective functions $K$ and highlights how feasibility of contraction depends on the choice of the ratio $\eta$, the geometric properties of $K$, and the strength of the entropy regularization (Theorem \ref{thm:convbycoupling}).

\end{itemize}
Moreover, in Appendix \ref{sec:appendixa}, we present an alternative approach using the averaging method to derive the leading term of the convergence rate in the finite-dimensional case when the interaction matrix dominates, a topic previously studied in \cite{wang2024local} through spectral analysis.
\subsection{Related works}
\paragraph{Two-timescale methods.} Including the classical two-timescale stochastic approximation \cite{borkar2008stochastic, konda2004convergence}, there have been many types of two-timescale learning methods applied across a variety of problems such as distributed optimization  \cite{lan2020communication, chang2020distributed}, finding fixed points in reinforcement learning \cite{daskalakis2020independent, dalal2020tale, angiuli2022unified, ding2020natural}, as well as finding global optimality in optimization as we mentioned previously. Separation of timescales often serves to enhance convergence behavior, facilitate empirical simulations \cite{berthier2024learning}, or even enable the discovery of distinct solutions by appropriately tuning learning rates \cite{angiuli2022unified}.

\paragraph{Hypocoercivity.} Originally developed for the analysis of kinetic equations such as the Boltzmann and Fokker-Planck equations, the hypocoercivity method \cite{villani2009hypocoercivity} has become a powerful tool in the study of stochastic and partial differential equations. It provides a framework for estimating energy dissipation by capturing the interplay between degenerate dissipation and mixing effects \cite{dolbeault2015hypocoercivity, grothaus2014hypocoercivity, roussel2018spectral, cao2023explicit}. This variational approach serves as an alternative to conventional methods based on spectral analysis, and we employ it to establish Theorem \ref{thm:hypocoercivity}.

\paragraph{Coupling methods.} Coupling is a powerful method in probability theory through which random variables can be compared with each other. Recently, variants of coupling methods have been developed to estimate contraction rates for Langevin dynamics \cite{eberle2016reflection, eberle2019couplings,schuh2024global} regarding different types of drifts, and later extended to sampling and optimization \cite{durmus2017nonasymptotic, cheng2018sharp, cheng2018underdamped, mou2022improved}.

\paragraph{Organization} 
The paper is organized as follows. In Section \ref{sec:setup}, we introduce the problem setup, notations, and informal statements of the main theorems. Section \ref{sec:finite-dim} investigates the convergence of the two-timescale GDA dynamics for the finite-dimensional quadratic game using a hypocoercivity approach. In Section \ref{sec:precond}, we discuss preconditioned saddle point problems, which help mitigate the effects of extreme values of the ratio $\eta$ and improve convergence. Finally, Section \ref{sec:infinite-dim} presents a detailed convergence analysis of the two-timescale mean-field GDA for a finite range of $\eta$ using the coupling method.

\section{Setup and main results}\label{sec:setup}
This section provides preliminary calculations, definitions, and notations used throughout the paper. Let $\mathcal{P}(\mathcal{X})$ denote the space of probability distributions over a compact manifold $\mathcal{X}$. $\|\cdot\|$ denotes the Euclidean norm for vectors, and $\|\cdot\|_F$ denotes the Frobenius norm of a matrix. We use $\kappa^{-}(r)=\max\{0,-\kappa(r)\}$ to denote the negative part of the function. Sometimes we use the notation $\dot{f}(t) = df/dt$.

\subsection{Finite dimensional quadratic games}
In the first part, we consider solving the quadratic game in the finite dimension.
    \begin{equation}\label{quadratic}
      \min_{x\in\Rm^n} \max_{y\in \Rm^m} K(x,y) = \min_{x\in\Rm^n} \max_{y\in \Rm^m}\Big\{ \frac{1}{2}x^{\top} Q x + x^{\top} P y - \frac{1}{2}y^{\top} R y\Big\}.
    \end{equation}
The quadratic game, which arises from control problems such as linear quadratic regulators (LQR), is a widely-studied subject. If we solve this zero-sum game by two-timescale gradient descent-ascent algorithm in the continuous time, then the dynamics would follow
    \begin{subequations}\label{alg:gda}
        \begin{align}
            \dot{x}(t) &= -\nabla_x K(x,y) = -Qx-Py,\\
            \dot{y}(t) &= \eta\nabla_y K(x,y) = -\eta Ry+\eta P^{\top} x.
        \end{align}
    \end{subequations}
    We may introduce a rescaled function $z(t) = \sqrt{\eta} x(t)$, so that the equations can be rewritten as 
    \begin{subequations}
        \begin{align}
            \dot{z}(t) &= -Qz-\sqrt{\eta}Py,\\
            \dot{y}(t) &= -\eta Ry+\sqrt{\eta} P^{\top} z.
        \end{align}
    \end{subequations}
    Taking $\phi(t) =[z(t), y(t)]^{\top}$, we have that 
    \begin{equation}\label{eqn:time_t}
    \begin{aligned}
        \dot{\phi}(t) = -\begin{bmatrix}
            Q & \sqrt{\eta}P\\
            -\sqrt{\eta}P^{\top} & \eta R
        \end{bmatrix} \phi(t) = -D \phi(t) - \sqrt{\eta}L \phi(t),
\end{aligned}
\end{equation}
with the symmetric matrix $D = \bigl[\begin{smallmatrix} Q & 0\\ 0 & \eta R \end{smallmatrix}\bigr]$ (while strictly speaking we shall use $D_{\eta}$, we suppress the subscript to keep notation simpler) and the skew-symmetric matrix $L = \bigl[\begin{smallmatrix} 0 & P\\ - P^{\top} & 0 \end{smallmatrix}\bigr]$.  
It is clear from linear ODE theory that the flow converges locally to a local saddle point $(x^*, y^*)$ if all eigenvalues of the matrix $D+\sqrt{\eta}L$ have strictly positive real parts:
\begin{equation}
    \mu_\eta: = \min_{\lambda \in \text{Sp}(D+\sqrt{\eta} L)} \text{Re}(\lambda) >0,
\end{equation}
where Sp$(\cdot)$ denotes the spectrum of a matrix. However, it is not obvious how the condition $\mu_\eta>0$ is affected by the two-timescale ratio $\eta$. While it is possible to conduct detailed spectral analysis as in \cite{wang2024local}, we choose to investigate the convergence rate dependence on $\eta$ from the hypocoercivity perspective \cite{villani2009hypocoercivity, dolbeault2015hypocoercivity} by constructing a Lyapunov function based on the interplay of symmetric and skew-symmetric matrices. Compared to spectral analysis, the analysis based on hypocoercivity is relatively simpler.
 
With the rescaled time $s=\sqrt{\eta} t$, we are able to obtain the following convergence result under appropriate assumptions (see Theorem \ref{thm:hypocoercivity} for rigorous statement), with some constants $C, \Lambda>0$,
\begin{equation*}
  \|\phi(s)\|^2 \leq C \exp\Big(-\Lambda \min\Big\{\sqrt{\eta}, \frac{1}{\sqrt{\eta}}\Big\} s \Big)\|\phi_0\|^2.
\end{equation*}
In particular, the convergence result suggests that optimal choice of $\eta$ is of order $1$, so that the timescales of dynamics of two components are comparable, rather than using $\eta \ll 1$ or $\eta \gg 1$ as in the quasi-static regime.

Returning to original variables of dynamics \eqref{alg:gda}, we get
\begin{equation}\label{eq:decaygame}
    \eta \norm{x(t)}^2 + \norm{y(t)}^2 \leq C e^{- \Lambda \min \{1, \eta \} t} \bigl(\eta \norm{x(0)}^2 + \norm{y(0)}^2 \bigr).
\end{equation}
While this seems to suggest choosing $\eta \gg 1$ to accelerate convergence, it is somewhat deceiving since this is only for the continuous-time dynamics, while usual discretization of \eqref{alg:gda} would require step size $O(1/\eta)$ when $\eta \gg 1$ due to the stiffness. The convergence of a preconditioned saddle point algorithm suggested by our analysis would be discussed in more details, see Section~\ref{sec:precond}.

\subsection{Infinite dimensional continuous games}

The infinite-dimensional entropy-regularized min-max problem is of the form
\begin{equation}\label{obj:entropy}
\begin{aligned}
    &\min_{p\in\mathcal{P}(\mathcal{X})} \max_{q\in\mathcal{P}(\mathcal{Y})} E_{\beta} (p, q): =\int_{\mathcal{Y}} \int_{\mathcal{X}} K(x,y) p(x) q(y) ~dx dy + \beta^{-1} H(p)- \beta^{-1} H(q)\\
    &\text{where}\quad H(p):=\int_{\mathcal{X}}\log \Big(\frac{dp}{dx}\Big) ~dp, \quad  H(q):=\int_{\mathcal{Y}}\log \Big(\frac{dq}{dy}\Big) ~dq,
\end{aligned}
\end{equation}
which is obtained from \eqref{eq:Kxy} by adding entropy regularization terms, where
$\beta > 0$ is a regularization parameter.  Thanks to the entropy
regularization, the objective energy functional $E_{\beta}(p, q)$ is
strongly convex in $p$ and strongly concave in $q$, thus by von
Neumann's minimax theorem
 \begin{equation}
   \min_{p\in\mathcal{P}(\mathcal{X})} \max_{q\in\mathcal{P}(\mathcal{Y})} E_{\beta} (p, q) =  \max_{q\in\mathcal{P}(\mathcal{Y})} \min_{p\in\mathcal{P}(\mathcal{X})} E_{\beta} (p, q).
 \end{equation}
 Theorem 4 in \cite{domingo2020mean} establishes that if in addition
 $K$ is continuous on $\mathcal{X}\times \mathcal{Y}$, a unique Nash
 equilibrium $(p^*, q^*)$ in the sense that
\begin{equation*}
    E_\beta(p^*,q)\leq E_\beta(p^*,q^*)\leq E_\beta(p,q^*),\quad \text{for all}~p\in\mathcal{P}(\mathcal{X}),~ q\in \mathcal{P}(\mathcal{Y})
\end{equation*}
exists and is the unique fixed point of 
\begin{equation}
    \begin{aligned}
        p(x) = \frac{1}{Z_p}\exp\Big(-\beta\int_{\mathcal{Y}} K(x,y) q(y) ~dy \Big),
        \quad  q(y) =  \frac{1}{Z_q}\exp\Big(\beta\int_{\mathcal{X}} K(x,y) p(x) ~dx \Big),
    \end{aligned}
\end{equation}
where $Z_p, Z_q$ are normalizing constants to make $p, q$ probability distributions.

 We focus on analyzing the dynamics of the entropy-regularized two-timescale gradient descent-ascent flow under the Wasserstein metric:
 \begin{subequations}\label{eqn:PDE}
     \begin{align}
         &\d_t p_t = \nabla_x \cdot\big(p_t  \int_{\mathcal{Y}}\nabla_x K(x,y) q_t(y) ~dy\big)+ \beta^{-1}\Delta_x p_t\\
         &\d_t q_t = \eta\Big(-\nabla_y \cdot\Big(q_t \int_{\mathcal{X}}\nabla_y K(x,y) p_t(x) ~dx\Big)+ \beta^{-1}\Delta_y q_t\Big),
     \end{align}
 \end{subequations}
which can be viewed as the mean-field limit ($N \to \infty$) of the two-timescale Langevin descent-ascent gradient flow 
\begin{subequations}\label{eq:meanfieldGDA}
    \begin{align}
        & dX_t^i = -\frac{1}{N}\sum_{j=1}^N \nabla_x K(X_t^i, Y_t^j) dt + \sqrt{2\beta^{-1}} d W_t^i\\
        & dY_t^i = \frac{\eta}{N}\sum_{j=1}^N \nabla_y K(X_t^j, Y_t^i) dt + \sqrt{2\eta\beta^{-1}} d B_t^i
    \end{align}
\end{subequations}
via $p_t = \frac{1}{N}\sum_{i=1}^N \delta_{X_t^i}, ~q_t = \frac{1}{N}\sum_{i=1}^N \delta_{Y_t^i}$, by taking the number of strategies $N$ to infinity. 

In the above dynamics, the parameter $\eta$ can be viewed as a ratio between the rate of the two dynamics.  We will adopt a mixed synchronous-reflection coupling approach, inspired by \cite{eberle2016reflection}, to establish a Wasserstein-1 convergence result. Given that the objective function is strongly convex-concave outside a local ball with radius $R$, we have 
 \begin{equation*}
     W_1((p_t,q_t), (p^*, q^*))\leq  Ce^{-c \min\{1, \eta\} t} W_1((p_0,q_0), (p^*, q^*)).
 \end{equation*}
We defer the details of the statement and required assumptions to Theorem \ref{thm:convbycoupling}. Notably, we obtain qualitative the same decay rate as the finite dimensional GDA dynamics \eqref{eq:decaygame}. 

This result, compared to \cite{conger2024coupled, cai2024convergence}, accommodates the scenario where $K$ is locally nonconvex-nonconcave, and extends the range of allowed $\eta$ value compared with \cite{lu2023two} which only addressed near quasi-static regime.



\section{Convergence of quadratic zero-sum games}\label{sec:finite-dim}
In this section, we focus on the two-timescale GDA dynamics (\ref{alg:gda}) for solving the quadratic zero-sum game (\ref{quadratic}). In particular, without using any spectral analysis tools, we want to investigate how the symmetric matrix $D$ together with the skew-symmetric matrix $L$ impact the convergence behavior as $\eta\ll1$ or $\eta\gg1$. By rescaling the equation (\ref{eqn:time_t}) with time $s= \sqrt{\eta} t$, we have
\begin{equation}\label{eqn:time_s}
    \frac{d}{d s} \phi = -\frac{1}{\sqrt{\eta}}D \phi- L\phi.
\end{equation}
The observation that
$D = \bigl[\begin{smallmatrix}
            Q & 0\\
            0 & \eta R
        \end{smallmatrix}\bigr]\sim \bigl[\begin{smallmatrix}
            Q & 0\\
            0 & 0
        \end{smallmatrix}\bigr]$ when $\eta\ll 1$ and $\frac{1}{\eta}D = \bigl[ \begin{smallmatrix}
            Q/\eta & 0\\
            0 &  R
        \end{smallmatrix} \bigr] \sim \bigl[\begin{smallmatrix}
            0 & 0\\
            0 & R
        \end{smallmatrix}\bigr]$ when $\eta\gg 1$ mimics the degenerate diffusion in the kinetic equation, which motivates us to take the hypocoercivity approach as in \cite{dolbeault2015hypocoercivity} for convergence analysis.

We start with a modified norm, which can be viewed as a Lyapunov function 
\begin{equation}
    H(\phi) = \frac{1}{2}\|\phi\|^2-\epsilon\langle M\phi, \phi\rangle
\end{equation}
with small $\epsilon\in (0,1)$ to be chosen, and the matrix $M$ is chosen to be
\begin{equation}\label{operator:M}
    M :=- ( I+(L\Pi)^{\top} L\Pi)^{-1}(L\Pi)^{\top}.
\end{equation}
Here $I \equiv I_{n+m}$ is an identity matrix, and $\Pi$ is a projection matrix onto kernel of a selected symmetric semi-definite matrix and satisfies the  algebraic assumption 
\begin{assumption}\label{assump:ag} The projection operator $\Pi$ satisfies
   \begin{equation}
    \Pi L \Pi = 0.
\end{equation} 
\end{assumption}
We specify what projection matrix $\Pi$ we will use depending on the cases. We write $\Pi_Q$ as projection onto the kernel of $Q$, so it functions as, for any $v\in \Rm^{n+m}$,
\begin{equation}
    \begin{bmatrix}
            Q & 0\\
            0 &  0
        \end{bmatrix}\Pi_Q v = 0.
\end{equation}
On the other hand, we write $\Pi_R$ to be projection onto the kernel of $R$, so that for any $v\in \Rm^{n+m}$,
\begin{equation}
    \begin{bmatrix}
            Q & 0\\
            0 &  R
        \end{bmatrix}\Pi_R v = 0.
\end{equation}
Based on the choice of $\Pi$, we have different lower bound assumptions for the symmetric matrix.
\begin{assumption}[microscopic coercivity]\label{assump:D}
 Given $S = \Bigl[\begin{smallmatrix}
            Q & 0\\
            0 &  R
        \end{smallmatrix}\Bigr]$, let $\lambda_Q$ and $\lambda_R$ be the least eigenvalues of $Q$ and $R$ respectively.  If $\Pi=\Pi_Q$, we have
        
        \begin{equation}
            \langle S\phi, \phi\rangle\geq \lambda_Q \|(I-\Pi_Q)\phi\|^2,
        \end{equation}
 and if $\Pi=\Pi_R$, we have
        
        \begin{equation}
           \langle S\phi, \phi\rangle\geq \lambda_R \|(I-\Pi_R)\phi\|^2.
        \end{equation}
\end{assumption}
Regarding the skew-symmetric matrix, we assume a universal lower bound for either $\Pi=\Pi_Q$ or $\Pi=\Pi_R$.
\begin{assumption}[macroscopic coercivity]\label{assump:L}
Let $\lambda_L$ be the minimum eigenvalue of $L^{\top}L = \Bigl[ \begin{smallmatrix}
            PP^\top & 0\\
            0 & P^\top P
        \end{smallmatrix}\Bigr]$, then
        \begin{equation}
            \|L\Pi\phi\|^2\geq \lambda_L \|\Pi\phi\|^2.
        \end{equation}
\end{assumption}

\begin{lemma}\label{lem:bounds}
    Given the operator $M$ in (\ref{operator:M}) and Assumption \ref{assump:ag}, we have the bounds
    \begin{equation}
         \|M\phi\|\leq \frac{1}{2}\|(I-\Pi)\phi\|,\quad \|LM\phi\|\leq\|(I-\Pi)\phi\|.
    \end{equation}
\end{lemma}
\begin{proof}
   For any $\phi$, we denote $f := M\phi$. Then we get
   \begin{equation}\label{eqn:aug261}
       (L\Pi)^{\top} \phi = -( I+(L\Pi)^{\top} L\Pi) f,
   \end{equation}
   which implies that
   \begin{equation}
     f = \Pi L \phi-\Pi L^\top L \Pi f.
   \end{equation}
Because both sides project onto the same vector space, we have $\Pi M = M$.

   Now taking the inner product of (\ref{eqn:aug261}) with $f$, we obtain that
   \begin{equation}\label{eqn:aug262}
  \|f\|^2 + \langle L\Pi f, L\Pi f\rangle =     \langle \phi, L\Pi f\rangle  = -\langle (I-\Pi)\phi, L\Pi f\rangle,
   \end{equation}
   as $\langle \Pi\phi, L\Pi f\rangle = 0$ due to Assumption \ref{assump:ag}. We apply the Young's inequality to the right side of (\ref{eqn:aug262}) to obtain the first bound,
   \begin{equation}
    \|f\|^2 + \|L\Pi f\|^2\leq \frac{1}{4}\|(I-\Pi)\phi\|^2 + \|L\Pi f\|^2.
   \end{equation}
   Moreover, to get the second bound in the result, we apply the Young's inequality to the right side  of (\ref{eqn:aug262}) again (with different constants), and drop $\|f\|^2$ on the left side, to arrive at
   \begin{equation}
       \|L\Pi f\|^2\leq\frac{1}{2}\|(I-\Pi)\phi\|^2 + \frac{1}{2}\|L\Pi f\|^2.
   \end{equation}
   Combining with $L\Pi f = L\Pi M\phi = LM\phi$, we then obtain the claim of the lemma.
\end{proof}
Furthermore, we also need some upper bound assumptions. 
\begin{assumption}\label{assump:MAupper}
 Given $S = \Bigl[\begin{smallmatrix}
            Q & 0\\
            0 &  R
        \end{smallmatrix}\Bigr]$, if $\Pi=\Pi_Q$, there exists $\Lambda_Q>0$ such that      
        \begin{equation}
           \langle M S\phi, \phi\rangle\leq \Lambda_Q\|(I-\Pi_Q)\phi\|\|\phi\|
        \end{equation}
 and if $\Pi=\Pi_R$, then there exists $\Lambda_R>0$ such that
        \begin{equation}
          \langle M S\phi, \phi\rangle\leq  \Lambda_R\|(I-\Pi_R)\phi\|\|\phi\|.
        \end{equation}
        Moreover, we can find a constant $C_M>0$ such that 
     \begin{equation}
              \langle M L(I-\Pi)\phi, \phi\rangle\leq C_M \|(I-\Pi)\phi\|\|\phi\|.
     \end{equation}
\end{assumption}
By Lemma \ref{lem:bounds}, we may choose small $\epsilon\in (0,1)$ to have norm equivalence between  $H(\phi)$ and $\|\phi\|^2$:
\begin{equation}\label{eqn:normequiv}
   \frac{1-\epsilon}{2}\|\phi\|^2\leq  H(\phi)\leq \frac{1+\epsilon}{2}\|\phi\|^2.
\end{equation}

With all assumptions stated, we are ready to state and prove the main result in this section. 
\begin{thm}\label{thm:hypocoercivity}
With Assumptions \ref{assump:ag}, \ref{assump:D}, \ref{assump:L}, and \ref{assump:MAupper}, we can find explicit constants $\Lambda, C>0$, depending on $\epsilon$,  to get that for all $s \geq 0$,
\begin{equation}
    \|\phi(s)\|^2 \leq C \exp\Big(-\Lambda \min\Big\{\sqrt{\eta}, \frac{1}{\sqrt{\eta}}\Big\} s \Big)\|\phi_0\|^2.
\end{equation}
\end{thm}
\begin{proof}
\textbf{Case 1: $\eta\ll 1$.} In this case, we consider 
$\tilde D  := \bigl[\begin{smallmatrix}
            Q & 0\\
            0 & 0
        \end{smallmatrix}\bigr]$
as an approximation to the diffusion matrix $D= \bigl[\begin{smallmatrix}
            Q & 0\\
            0 & \eta R
        \end{smallmatrix}\bigr]$,
       We can find a constant $C_R>0$ such that
        \begin{equation}\label{eqn:err1}
            \|\tilde D- D\|_F\leq C_R \eta.
        \end{equation}
With $\tilde D$ in mind, we choose $\Pi = \Pi_Q$ satisfying Assumption \ref{assump:ag} that projects onto kernel of  $\tilde D$, and write $H(\phi) = H_Q(\phi)$ accordingly. Taking the derivative of $H(\phi)$ in time $s$, by the dynamics (\ref{eqn:time_s}), we have
\begin{equation}
    \begin{aligned}
        \frac{d}{ds} H_Q(\phi) &= \langle \phi, \phi_s\rangle-\epsilon\langle M\phi_s, \phi\rangle-\epsilon\langle M\phi, \phi_s\rangle\\
        &= -\frac{1}{\sqrt{\eta}}\langle \phi, D\phi\rangle + \frac{\epsilon}{\sqrt{\eta}}\langle MD\phi, \phi\rangle+\epsilon\langle M L\phi, \phi\rangle\\
        &\quad +\frac{\epsilon}{\sqrt{\eta}}\langle M\phi, D\phi\rangle+\epsilon\langle M\phi, L \phi\rangle.
    \end{aligned}
\end{equation}
For $\langle \phi, D\phi\rangle$, we utilize Assumption \ref{assump:D} and (\ref{eqn:err1}) to obtain that
\begin{equation}
\begin{aligned}
    -\frac{1}{\sqrt{\eta}}\langle   D \phi, \phi\rangle& \leq  -\frac{1}{\sqrt{\eta}}\langle  \tilde D \phi, \phi\rangle +\frac{1}{\sqrt{\eta}}\big|\langle (D - \tilde D)\phi, \phi\rangle \big|\\
    &\leq -\frac{\lambda_Q}{\sqrt{\eta}}\|(I-\Pi_Q)\phi\|^2+C_R \sqrt{\eta} \|\phi\|^2.
\end{aligned}
\end{equation}
For $\langle MD\phi, \phi\rangle$, by the property that $\Pi_Q M= M$ and Assumption \ref{assump:ag}, we observe that
\begin{equation}
   \langle M \tilde{D}\phi, \phi \rangle = \langle M\phi, \tilde{D}\phi\rangle = \langle \Pi_Q M\phi, \tilde{D}\phi\rangle=\langle \tilde{D} \Pi_Q M\phi, \phi\rangle = 0.
\end{equation}
Thus
\begin{equation}
    \frac{\epsilon}{\sqrt{\eta}}\langle MD\phi, \phi\rangle\leq \frac{\epsilon}{\sqrt{\eta}}\big|\langle M(D-\tilde{D})\phi, \phi\rangle\big|\leq \epsilon C_R\sqrt{\eta}\|\phi\|^2.
\end{equation}
As for $\langle M L\phi, \phi\rangle$, we may split
\begin{equation}
    \begin{aligned}
        \langle M L\phi, \phi\rangle = \langle M L\Pi\phi, \phi\rangle+ \langle M L(I-\Pi)\phi, \phi\rangle,
    \end{aligned}
\end{equation}
and treat the first part by using Assumption \ref{assump:L}:
\begin{equation}
    \langle M L\Pi\phi, \phi\rangle = -\langle( I+(L\Pi)^{\top} L\Pi)^{-1}(L\Pi)^{\top}(L\Pi)\phi, \phi\rangle \leq -\frac{\lambda_L}{1+\lambda_L} \|\Pi\phi\|^2 = -\frac{\lambda_L}{1+\lambda_L} \|\Pi_Q\phi\|^2.
\end{equation}
As for the second part, we have the bound  by  Assumption \ref{assump:MAupper},
\begin{equation}
     \langle M L(I-\Pi_Q)\phi, \phi\rangle\leq C_M \|(I-\Pi_Q)\phi\|\|\phi\|.
\end{equation}
For $\langle M\phi, D\phi\rangle$, we then use bounds in Lemma \ref{lem:bounds} and (\ref{eqn:err1}) to get that
\begin{equation}
\begin{aligned}
   \frac{\epsilon}{\sqrt{\eta}} \langle M\phi, D\phi\rangle&=\frac{\epsilon}{\sqrt{\eta}} \langle M\phi, \tilde{D}\phi\rangle+\frac{\epsilon}{\sqrt{\eta}} \langle M\phi, (D-\tilde{D})\phi\rangle\\
   &\leq \frac{\epsilon}{\sqrt{\eta}}\|M\phi\|\Big(\|\tilde{D}\phi\|+\|(D-\tilde{D})\phi\|\Big) \\
   &\leq \frac{\epsilon\Lambda_Q}{2\sqrt{\eta}}\|(I-\Pi_Q)\phi\|\|\phi\| + \frac{\epsilon C_R\sqrt{\eta}}{2} \|(I-\Pi_Q)\phi\| \|\phi\|.
\end{aligned}
\end{equation}
Lastly,  for $\langle M\phi, L \phi\rangle$, by Lemma \ref{lem:bounds}, we obtain that
\begin{equation}
    \langle M\phi, L \phi\rangle = -\langle LM\phi,  \phi\rangle\leq  \|(I-\Pi_Q)\phi\|\|\phi\|.
\end{equation}
Combining all bounds above together, we have that
\begin{equation}
\begin{aligned}
    \frac{d}{ds} H_Q(\phi) &\leq -\frac{\lambda_Q}{\sqrt{\eta}}\|(I-\Pi_Q)\phi\|^2-\frac{\epsilon\lambda_L}{1+\lambda_L} \|\Pi_Q\phi\|^2\\
    &\quad + \epsilon\Big(\frac{\Lambda_Q}{2\sqrt{\eta}}+\frac{C_R\sqrt{\eta}}{2}+1+C_M\Big)\|(I-\Pi_Q)\phi\|\|\phi\| + (1+\epsilon) C_R \sqrt{\eta} \|\phi\|^2\\
    & = -\begin{bmatrix}
            \|(I-\Pi_Q)\phi\| \\
            \|\Pi_Q\phi\|
        \end{bmatrix}^\top \begin{bmatrix}
            S_{++} & S_{+-}/2\\
          S_{+-}/2 & S_{--}
        \end{bmatrix} \begin{bmatrix}
            \|(I-\Pi_Q)\phi\| \\
            \|\Pi_Q\phi\| 
        \end{bmatrix}+(1+\epsilon)C_R \sqrt{\eta}\|\phi\|^2,
\end{aligned}
\end{equation}
with 
\begin{equation}
    S_{++} = \frac{\lambda_Q}{\sqrt{\eta}}, \quad S_{--} = \frac{\epsilon\lambda_L}{1+\lambda_L} , \quad S_{+-} = - \epsilon\Big(\frac{\Lambda_Q}{2\sqrt{\eta}}+\frac{C_R\sqrt{\eta}}{2}+1+C_M\Big)
\end{equation}
and the smallest eigenvalue of the matrix $\bigl[\begin{smallmatrix}
            S_{++} & S_{+-}/2\\
          S_{+-}/2 & S_{--}
        \end{smallmatrix} \bigr]$ is
\begin{equation}
    \lambda_{\min} = \frac{1}{2}\Big(S_{++} + S_{--}-\sqrt{(S_{++}-S_{--})^2+(S_{+-})^2}\Big).
\end{equation}
Note that we consider the regime where $\eta\ll1$. In order to make $\lambda_{\min}>0$,  we take $\epsilon \sim \sqrt{\eta}$ so that for sufficiently small $\eta$, we can find a constant $c>0$ to get
\begin{equation}
     \frac{d}{ds} H_Q(\phi) \leq -c\sqrt{\eta} \|\phi\|^2\leq -\frac{2c\sqrt{\eta}}{1+\epsilon} H_Q(\phi)
\end{equation}
by the norm equivalence (\ref{eqn:normequiv}), which also implies that for $\eta\ll 1$,
\begin{equation}
    H_Q(\phi(s))\leq \exp\Big(\frac{-2c\sqrt{\eta}s}{1+\epsilon}\Big) H_Q(\phi_0) \quad \Longrightarrow \quad \|\phi(s)\|^2 \leq \frac{1+\epsilon}{1-\epsilon}\exp\Big(\frac{-2c\sqrt{\eta}s}{1+\epsilon}\Big) \|\phi_0\|^2.
\end{equation}
\textbf{Case 2: $\eta\gg 1$.}
We first rewrite the dynamics (\ref{eqn:time_s}) as
\begin{equation}
     \phi_s = -\sqrt{\eta }B \phi- L\phi
\end{equation}
with the re-scaled diffusion matrix  
$B= \frac{1}{\eta} D=  \bigl[ \begin{smallmatrix}
            Q/\eta & 0\\
            0 & R
 \end{smallmatrix}\bigr]$. We take 
$\tilde B  := \bigl[\begin{smallmatrix}
            0 & 0\\
            0 & R
        \end{smallmatrix}\bigr]$ to compare with $B$, and one can find a constant $C_Q>0$ such that 
        \begin{equation}
            \|\tilde B- B\|_F\leq C_Q \eta^{-1}.
        \end{equation}
With $\tilde B$ in mind, we choose $\Pi = \Pi_R$ satisfying Assumption \ref{assump:ag} that projects onto kernel of  $\tilde B$, and write $H(\phi) = H_R(\phi)$ accordingly. Taking derivative of $H(\phi)$ in time $s$, we then get
\begin{equation}
    \begin{aligned}
        \frac{d}{ds} H_R(\phi) &= \langle \phi, \phi_s\rangle-\epsilon\langle M\phi_s, \phi\rangle-\epsilon\langle M\phi, \phi_s\rangle\\
        &= -\sqrt{\eta}\langle \phi, B\phi\rangle + \epsilon\sqrt{\eta}\langle M B\phi, \phi\rangle+\epsilon\langle M L\phi, \phi\rangle\\
        &\quad +\epsilon\sqrt{\eta}\langle M\phi, B\phi\rangle+\epsilon\langle M\phi, L \phi\rangle.
    \end{aligned}
\end{equation}
The estimates of each term are similar to what we have done in Case 1. Omitting those details, we will eventually get the bound that
\begin{equation}
\begin{aligned}
    \frac{d}{ds} H_R(\phi) &\leq -\lambda_R\sqrt{\eta}\|(I-\Pi_R)\phi\|^2-\frac{\epsilon\lambda_L}{1+\lambda_L} \|\Pi_R\phi\|^2\\
    &\quad+\epsilon\Big(\frac{\Lambda_R \sqrt{\eta}}{2}+\frac{C_R}{2\sqrt{\eta}}+1+C_M\Big)\|(I-\Pi_R)\phi\|\|\phi\|
  + \frac{(1+\epsilon)C_Q}{\sqrt{\eta}} \|\phi\|^2\\
    & = -\begin{bmatrix}
            \|(I-\Pi_Q)\phi\| \\
            \|\Pi_Q\phi\|
        \end{bmatrix}^\top \begin{bmatrix}
            S_{++} & S_{+-}/2\\
          S_{+-}/2 & S_{--}
        \end{bmatrix} \begin{bmatrix}
            \|(I-\Pi_Q)\phi\| \\
            \|\Pi_Q\phi\| 
        \end{bmatrix}+ \frac{(1+\epsilon) C_Q}{\sqrt{\eta}} \|\phi\|^2,
\end{aligned}
\end{equation}
with \begin{equation}
    S_{++} = \lambda_R\sqrt{\eta}, \quad S_{--} = \frac{\epsilon\lambda_L}{1+\lambda_L} , \quad S_{+-} = - \epsilon\Big(\frac{\Lambda_R \sqrt{\eta}}{2}+\frac{C_R}{2\sqrt{\eta}}+1+C_M\Big)
\end{equation}
and the smallest eigenvalue of this matrix is
\begin{equation}
    \lambda_{\min} = \frac{1}{2}\Big(S_{++} + S_{--}-\sqrt{(S_{++}-S_{--})^2+(S_{+-})^2}\Big).
\end{equation}
Since we consider the regime where $\eta\gg 1$, in order to make $\lambda_{\min}>0$ we take $\epsilon \sim 1/\sqrt{\eta}$ so that there exists a constant $c'>0$,
\begin{equation}
     \frac{d}{ds} H_R(\phi) \leq -\frac{c'}{\sqrt{\eta}} \|\phi\|^2\leq -\frac{2c'}{(1+\epsilon)\sqrt{\eta}} H_R(\phi).
\end{equation}
Accompanied with the norm equivalence (\ref{eqn:normequiv}), we conclude that for $\eta\gg1$, 
\begin{equation}
    H_R(\phi(s))\leq \exp\Big(\frac{-2c's}{(1+\epsilon)\sqrt{\eta}}\Big) H_R(\phi_0) \quad \Longrightarrow \quad \|\phi(s)\|^2 \leq \frac{1+\epsilon}{1-\epsilon}\exp\Big(\frac{-2c's}{(1+\epsilon)\sqrt{\eta}}\Big) \|\phi_0\|^2.
\end{equation}
\end{proof}
In the Figure \ref{fig:psi},  we numerically verify Theorem \ref{thm:hypocoercivity} by  demonstrating that the optimal choice of
$\eta$ for minimizing the least eigenvalue is of order $1$, away from the quasi-static regimes.
\begin{figure}[h!]
\centering
\includegraphics[width=.6\textwidth, height=.4\textwidth]{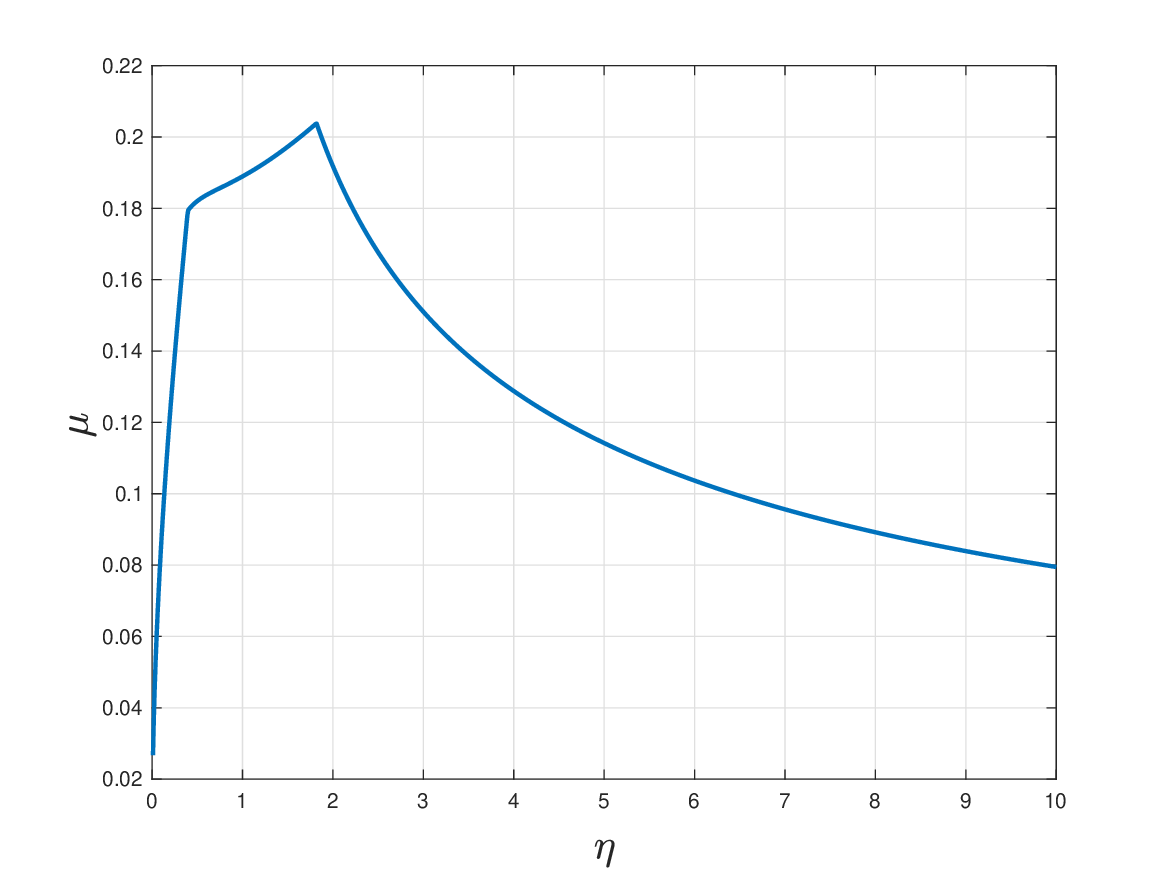}
 \caption{An illustration of how the least eigenvalue of $\frac{1}{\sqrt{\eta}} D+L$ depends on $\eta$ (chosen to be $\eta = 0.01:0.01:10$). We randomly generate $10\times10$ symmetric semi-definite matrices $Q, R$ to create $D=\bigl[\begin{smallmatrix}
            Q & 0\\
            0 & \eta R
        \end{smallmatrix}\bigr]$ and $10\times 10$ matrix $P$ to create $L=\bigl[\begin{smallmatrix}
            0 & P\\
            - P^{\top} & 0
        \end{smallmatrix}\bigr]$. The kink points in the plot are caused by eigenvalue crossings. }
 \label{fig:psi}
\end{figure}

\section{Preconditioned saddle point problem}\label{sec:precond}

Theorem~\ref{thm:hypocoercivity} establishes the convergence behavior of continuous-time dynamics \eqref{eqn:time_t}, while in practice, we would need a numerical discretization to \eqref{eqn:time_t} of the form 
\begin{equation}\label{eqn:numer}
    \phi_{k+1}-\phi_k = -\rho\Big(D +\sqrt{\eta} L\Big)\phi_k,
\end{equation}
where $\rho$ is the step size. The condition number of the matrix $D+\sqrt{\eta}L$ can significantly limit the choice of step size $\rho>0$ for convergence, especially in the case  where $\eta\gg1$ or $\eta\ll 1$. The condition number, defined as the ratio of the maximal to minimal singular values of the matrix, is at least of the order
\begin{equation}
   \kappa\Big(D +\sqrt{\eta}L\Big)= \kappa\Big(\frac{1}{\sqrt{\eta}}D +L\Big)\geq O(\max\{\eta, 1/\eta\}).
\end{equation}
This is because $\sigma_{\max}\Big(\frac{1}{\sqrt{\eta}}D +L\Big) = \Big\|\frac{1}{\sqrt{\eta}}D +L\Big\|_2\geq \max\{\xi^\top\Big(\frac{1}{\sqrt{\eta}}D +L\Big)\xi, \xi\in \R^{m+n}\}= O(\max\{\sqrt{\eta},1/\sqrt{\eta}\})$, and  $\sigma_{\min}\Big(\frac{1}{\sqrt{\eta}}D +L\Big)=O(\min\{\sqrt{\eta},1/\sqrt{\eta}\})$, as inferred from Theorem \ref{thm:hypocoercivity}. Thus, In practice, it is crucial to include a preconditioner for the matrix $D +\sqrt{\eta}L$ to mitigate the impact of extreme values of $\eta$. Specifically, we aim to find a preconditioner $\mathcal{P}$ such that the numerical updates, with $\xi$ being the transformed vector $\phi$, follow as
\begin{equation}\label{eqn:target}
    \xi_{k+1}-\xi_k = -\rho\mathcal{P}\Big(D +\sqrt{\eta} L\Big)\xi_k,
\end{equation}
so that the choice of $\rho$ does not depend on $\eta$.

Suppose $Q$ is symmetric positive definite, ensuring that it is invertible. We then consider the following equation via block Gaussian elimination:
\begin{equation}
\begin{bmatrix}
        I & 0\\
        \sqrt{\eta} P^\top Q^{-1} & I
    \end{bmatrix}
    \begin{bmatrix}
        Q & \sqrt{\eta} P\\
        -\sqrt{\eta} P^\top & \eta R
    \end{bmatrix}
    \begin{bmatrix}
        I & -\sqrt{\eta} Q^{-1}P\\
        0& I
    \end{bmatrix} = \begin{bmatrix}
        Q & 0\\
        0 & \eta( R+P^\top Q^{-1} P)
    \end{bmatrix}.
\end{equation}
The condition number on the right-hand side will not depend on $\eta$ if we further multiply it by $S=\begin{bmatrix}
        I & 0\\
        0 & I/\eta 
    \end{bmatrix}$.
We define $M = \begin{bmatrix}
        I & 0\\
        \sqrt{\eta} P^\top Q^{-1} & I
    \end{bmatrix}, N = \begin{bmatrix}
        I & -\sqrt{\eta} Q^{-1}P\\
        0& I
    \end{bmatrix}$, and introduce $\xi_k = N^{-1}\phi_k$. Thus the numerical update (\ref{eqn:numer}) with a left preconditioner $SM$ becomes
    \begin{equation}
        N\xi_{k+1} - N\xi_k = -\rho S M (D+\sqrt{\eta}L) N \xi_k = -\rho \begin{bmatrix}
        Q & 0\\
        0 &  R+P^\top Q^{-1} P
    \end{bmatrix}\xi_k, \quad 
    \end{equation}
Therefore, we take $\mathcal{P} = N^{-1}S M$, so that (\ref{eqn:target}) is
\begin{equation}
    \xi_{k+1} - \xi_k = -\rho N^{-1}\begin{bmatrix}
        Q & 0\\
        0 &  R+P^\top Q^{-1} P
    \end{bmatrix}\xi_k, \quad N^{-1} = \begin{bmatrix}
        I & \sqrt{\eta} Q^{-1}P\\
        0& I
    \end{bmatrix}.
\end{equation}
From the above, we conclude that by taking the preconditioner $\mathcal{P} = N^{-1}SM$, the condition number 
\begin{equation}
    \kappa \Big(N^{-1}\begin{bmatrix}
        Q & 0\\
        0 &  R+P^\top Q^{-1} P 
    \end{bmatrix}\Big)\leq  \kappa\Big(\begin{bmatrix}
        Q & 0\\
        0 &  R+P^\top Q^{-1} P 
    \end{bmatrix}\Big)
\end{equation}
does not depend on $\eta$.

Before moving toward convergence results, we note that $\phi = [\sqrt{\eta} x, y]^\top$, and $\xi = N^{-1}\phi = [\sqrt{\eta}(x+Q^{-1}Py), y]^\top $. Throughout the remaining content, we denote $T=\begin{bmatrix}
        Q & 0\\
        0 &  R+P^\top Q^{-1} P 
    \end{bmatrix} $.

\begin{claim}
    All eigenvalues of $N^{-1}T$ are real and nonnegative.
\end{claim}
\begin{proof}
    Note that $N^{-1}T = \begin{bmatrix}
        Q & \sqrt{\eta} Q^{-1}P(R+P^\top Q^{-1} P)\\
        0 &  R+P^\top Q^{-1} P 
    \end{bmatrix}$. We consider its eigenvalue $\lambda$ with associated eigenvector $[u;v]^\top$, which gives equations
    \begin{align}
        Qu + \sqrt{\eta} Q^{-1}P(R+P^\top Q^{-1} P)v &= \lambda u,\\
        (R+P^\top Q^{-1} P)v &= \lambda v.
    \end{align}
    The second equation, since $R+P^\top Q^{-1} P$ is symmetric, implies that if $v\neq 0 $, then $\lambda$ must be real and nonnegative. Otherwise, if $\lambda$ is complex, then $v=0$. The second case cannot be true because if $v=0$, we have $Qu = \lambda u$, which implies either $u=0$, or $\lambda$ is real and nonnegative. Therefore, only the first case exists for nontrivial eigenvectors.
\end{proof}
\begin{lemma}
    By choosing $\rho  = \frac{\lambda_{\min}(N^{-1}T+(N^{-1}T)^\top)}{2\sigma_{\max}^2(N^{-1}T)}$, the iterates of (\ref{eqn:target}) shrink as 
    \begin{equation}
        \|\xi_{k+1}\|\leq \sqrt{1- \frac{\lambda_{\min}^2(N^{-1}T+(N^{-1}T)^\top)}{4\sigma_{\max}^2(N^{-1}T)}}\|\xi_k\|.
    \end{equation}
\end{lemma}
\begin{proof}
From (\ref{eqn:target}), we have
\begin{equation}
    \begin{aligned}
          \|\xi_{k+1}\|^2 &= \|(I-\rho \mathcal{P}(D+\sqrt{\eta}L))\xi_{k}\|^2 = \|(I-\rho N^{-1}T)\xi_{k}\|^2 \\
          & = \|\xi_k\|^2-\rho \langle\xi_k,N^{-1}T \xi_k\rangle- \rho \langle N^{-1}T\xi_k, \xi_k\rangle+ \rho^2 \|N^{-1}T\xi_k\|^2\\
          &\leq (1+\rho^2\sigma_{\max}^2(N^{-1}T))\|\xi_k\|^2 - \rho\lambda_{\min}(N^{-1}T+(N^{-1}T)^\top)\|\xi_k\|^2\\
          &= \Big(1- \frac{\lambda_{\min}^2(N^{-1}T+(N^{-1}T)^\top)}{4\sigma_{\max}^2(N^{-1}T)}\Big)\|\xi_k\|^2,
    \end{aligned}
\end{equation}
and the last line is achieved by taking $\rho = \frac{\lambda_{\min}(N^{-1}T+(N^{-1}T)^\top)}{2\sigma_{\max}^2(N^{-1}T)}$. 
\end{proof}
\begin{remark}
    We observe that 
\begin{equation}
    \frac{\lambda_{\min}(N^{-1}T+(N^{-1}T)^\top)}{2\sigma_{\max}(N^{-1}T)}\leq \frac{\sigma_{\min}(N^{-1}T)}{\sigma_{\max}(N^{-1}T)} = \frac{1}{\kappa(N^{-1}T)},
\end{equation}
Thus, for the iteration $\|\xi_{k+1}\|\leq c \|\xi_k\|$, the convergence rate $c = \sqrt{1- \frac{\lambda_{\min}^2(N^{-1}T+(N^{-1}T)^\top)}{4\sigma_{\max}^2(N^{-1}T)}}\geq \sqrt{1-1/\kappa^2(N^{-1}T)}$. Therefore, with the preconditioner $\mathcal{P}$, the condition number of the resulting matrix $N^{-1}T$  does not depend on the learning rate ratio $\eta$, and allows for a faster convergence rate when $\eta$ is extreme.
\end{remark}

Lastly, for completeness, we examine the convergence rate of the preconditioned continuous dynamics. With the properly chosen preconditioner $\mathcal{P} = N^{-1}SM$, the convergence rate of the dynamics $\dot{\xi}(t) = -N^{-1}T \xi(t)$ is also improved for extreme $\eta$, compared with Theorem~\ref{thm:hypocoercivity}. 
\begin{lemma}
   Suppose that both $Q$ and $R$ are symmetric positive definite with least eigenvalues of order $O(1)$, then for all choices of $\eta$, there exists a constant $C>0$ such that $\lambda_{\min}(N^{-1}T)\geq C\max\{1,\sqrt{\eta}\}$. Therefore, for the continuous-time equation $\dot{\xi}(t) = -N^{-1}T \xi(t)$, $$\|\xi(t)\|\leq \|\xi_0\|e^{-C\max\{1,\sqrt{\eta}\}t}.$$ 
\end{lemma}
\begin{proof}
    All we need to show is that when $\eta$ is extreme, the least eigenvalue of $N^{-1}T$ remains at least of order $O(1)$. 
    
    We first consider the case $\eta\ll1$. Note that
    $$N^{-1}T = \begin{bmatrix}
        Q & \sqrt{\eta} Q^{-1}P(R+P^\top Q^{-1} P)\\
        0 &  R+P^\top Q^{-1} P 
    \end{bmatrix} = T + E,\quad E = \begin{bmatrix}
        0 & \sqrt{\eta} Q^{-1}P(R+P^\top Q^{-1} P)\\
        0 &  0
    \end{bmatrix},$$
    and by Weyl's inequality, $|\sigma_{\min}(N^{-1}T)-\sigma_{\min}(T)|\leq |\sigma_{\max}(E)|= O(\sqrt{\eta})$, so we have $\sigma_{\min}(N^{-1}T)\geq O(1)$ for sufficiently small $\eta$.

    In the case $\eta\gg 1$, $\sigma_{\max}(N^{-1}T)= \max\{ \|N^{-1}T\xi\|_2:\|\xi\|=1\} = O(\sqrt{\eta})$. The $O(1)$ order condition number of $N^{-1} T$ implies that $\sigma_{\min}(N^{-1}T)=O(\sqrt{\eta})$.

    By the relation $\lambda_{\min}(N^{-1}T)\geq \sigma_{\min}(N^{-1}T)$, we conclude that there exists a constant $C>0$ such that $\lambda_{\min}(N^{-1}T)\geq C\max\{1,\sqrt{\eta}\}$.
\end{proof}

\section{Convergence of the continuous game}\label{sec:infinite-dim}
In this section, we analyze the convergence of the two-timescale gradient descent-ascent flow under the Wasserstein metric (\ref{eqn:PDE}) for finite $\eta$. We start with assumptions on $K$. The following assumes that the gradients of $K$ are Lipschitz continuous, and $K$ is strongly convex-concave outside a ball (when two points are more than $R>0$ apart). Note that $K$ is allowed to be locally nonconvex-nonconcave. 
\begin{assumption}\label{assump:K}
    There exist $\kappa_x, \kappa_y, m_x, m_y>0$ such that for all $x\in\mathcal{X}, y\in\mathcal{Y}$, we have
    \begin{equation}\label{assump:cond3_1}
        \begin{aligned}
            &\langle \nabla_x K(x_1, y)- \nabla_x K(x_2, y), x_1-x_2\rangle \geq \kappa_x \|x_1-x_2\|^2, ~\text{for all}~\|x_1 -x_2\|\geq R, y\in\mathcal{Y},\\
            &\langle \nabla_y K(x, y_1)- \nabla_y K(x, y_2), y_1-y_2\rangle \leq -\kappa_y \|y_1-y_2\|^2, ~\text{for all}~x\in\mathcal{X}, \|y_1-y_2\|\geq R,
        \end{aligned}
    \end{equation}
    and 
        \begin{equation}
        \begin{aligned}
            &\langle \nabla_x K(x_1, y)- \nabla_x K(x_2, y), x_1-x_2\rangle \geq -m_x \|x_1-x_2\|^2, ~\text{for all}~\|x_1 -x_2\|<R, y\in\mathcal{Y},\\
            &\langle \nabla_y K(x, y_1)- \nabla_y K(x, y_2), y_1-y_2\rangle \leq m_y \|y_1-y_2\|^2, ~\text{for all}~x\in\mathcal{X}, \|y_1-y_2\|< R,
        \end{aligned}
    \end{equation}
Moreover, there exist positive constants $L_{X}, L_{Y}, l_{X}, l_{Y}$ such that
\begin{equation}\label{assump:cond3_2}
    \begin{aligned}
      &\|\nabla_x K(x_1, y_1)- \nabla_x K(x_2, y_2)\| \leq l_X\|x_1-x_2\|+L_Y \|y_1-y_2\|, \\
       &\|\nabla_y K(x_1, y_1)- \nabla_y K(x_2, y_2)\| \leq L_X \|x_1-x_2\|+l_Y \|y_1-y_2\|,
    \end{aligned}
\end{equation}
for all $x_1,x_2\in\mathcal{X}, y_1, y_2\in\mathcal{Y}$.
\end{assumption}
We look at the two-timescale gradient descent-ascent Langevin dynamics, by taking the realizations $X_t$ from the marginal distribution $ p_t = \text{Law}(X_t)$ and $Y_t\sim q_t = \text{Law}(Y_t)$, and they follow
\begin{subequations}\label{eqn:121}
\begin{align}
    dX_t &=-\int_{\mathcal{Y}} \nabla_x K(X_t,y) q_t(y) dy ~dt + \sqrt{2\beta^{-1}} dW_t,\\
    dY_t&=\eta \int_{\mathcal{X}} \nabla_y K(x,Y_t) p_t(x) dx~ dt + \sqrt{2\eta\beta^{-1}} dB_t,
\end{align}
\end{subequations}
where $W, B$ are independent Brownian motions. The motivation for us to look into the coupled overdamped Langevin equations is to take the reflection-synchronous coupling approach introduced by \cite{eberle2016reflection}. Let us consider the coupling for diffusion processes $(X_t, Y_t) $ and $(\tilde{X}_t, \tilde{Y}_t)$, and we use the notation $Z_t = X_t-\tilde{X}_t, Q_t = Y_t- \tilde{Y}_t$. 

Let us fix $\delta>0$ and suppose that we have Lipschitz continuous functions $\frc_i(Z_t, Q_t)$ and $\fsc_i(Z_t, Q_t), i=1,2$ satisfy
\begin{equation}
\begin{aligned}
    &\frc_i(Z_t, Q_t)^2 + \fsc_i(Z_t, Q_t)^2 = 1,\\
    &\frc_1(Z_t, Q_t) = 1 ~\text{if}~\|Z_t\|\geq \delta, \quad  \frc_1(Z_t, Q_t) = 0 ~\text{if}~\|Z_t\|\leq \delta/2,\\
    &\frc_2(Z_t, Q_t) = 1 ~\text{if}~\|Q_t\|\geq \delta, \quad  \frc_2(Z_t, Q_t) = 0 ~\text{if}~\|Q_t\|\leq \delta/2.
\end{aligned}
\end{equation}
We will often write $\frc \equiv \frc(z,q), \fsc \equiv \fsc(z,q)$ for notation simplicity. Then over $(\mathcal{X}\times \mathcal{Y})^2$, a coupling of two solutions to (\ref{eqn:121}) is given by 
\begin{equation}\label{oct291}
\begin{aligned}
    d X_t &=-\int_{\mathcal{Y}} \nabla_x K(X_t,y) q_t(y) dy~dt + \sqrt{2\beta^{-1}} \frc_1 dW_t^{\frc}+ \sqrt{2\beta^{-1}} \fsc_1dW_t^{\fsc},\\
    dY_t&=\eta \int_{\mathcal{X}} \nabla_y K(x,Y_t) p_t(x) dx~dt + \sqrt{2\eta\beta^{-1}}\frc_2 d B_t^{\frc}+ \sqrt{2\eta\beta^{-1}}\fsc_2 d B_t^{\fsc},\\
    d\tilde{X}_t &=-\int_{\mathcal{Y}} \nabla_x K(\tilde{X}_t,\tilde{y}) \tilde{q}_t(\tilde{y}) d\tilde{y}~dt + \sqrt{2\beta^{-1}} \frc_1(I-2e_t^1 e_t^{1,\top}) dW_t^{\frc}+ \sqrt{2\beta^{-1}} \fsc_1dW_t^{\fsc},\\
    d\tilde{Y}_t&=\eta \int_{\mathcal{X}} \nabla_y K(\tilde{x},\tilde{Y}_t) \tilde{p}_t(\tilde{x}) d\tilde{x}~dt + \sqrt{2\eta\beta^{-1}}\frc_2 (I-2e_t^2 e_t^{2,\top}) d B_t^{\frc}+ \sqrt{2\eta\beta^{-1}}\fsc_2 d B_t^{\fsc},
\end{aligned}
\end{equation}
with independent Brownian motions $W^{\frc}, W^{\fsc}, B^{\frc}, B^{\fsc}$, and initial distribution $(X_0, Y_0)\sim \mu, (\tilde{X}_0, \tilde{Y}_0)\sim \nu$. Here $e_t e_t^\top$ is the orthogonal projection onto the unit vector as
\begin{equation}
    e_t^1  =\begin{cases} Z_t/\|Z_t\|, ~\text{if}~Z_t\neq 0,\\ u_1,\quad  \quad\quad\text{if}~Z_t=0,
    \end{cases} e_t^2 =\begin{cases} Q_t/\|Q_t\|, ~\text{if}~Q_t\neq 0,\\ u_2,\quad  \quad\quad \text{if}~Q_t=0,
    \end{cases}
\end{equation}
and $u_1\in\mathcal{X}, u_2\in\mathcal{Y}$ are some arbitrary fixed unit vectors.

Here, we emphasize that applying a reflection coupling individually to each process may not work, as the realization and its reflection can move apart after the coupling time for each process due to the interaction between the two processes. To address this, we employ a mixed reflection-synchronous coupling with regularized reflection functions, avoiding reliance on the coupling time in the local regime. Eventually, we will pass the distance constant $\delta$ to zero to remove the additional error term arises from this construction, to obtain the main result, Theorem \ref{thm:convbycoupling}.

In order to construct an appropriate metric, we define 
\begin{equation}\label{def:kappa}
    \begin{aligned}
        \kappa_x(r,t) &:= r^{-2} \inf\Bigl\{\int_{\mathcal{Y}}(x-x')\cdot \bigl(\nabla_x K(x, y)- \nabla_x K(x', y)\bigr) q_t(y)dy ~\text{s.t.}~\|x-x'\|=r\Bigr\},\\
         \kappa_y(r,t) &:= r^{-2} \inf\Bigl\{-\int_{\mathcal{X}}(y-y')\cdot \bigl(\nabla_y K(x, y)- \nabla_y K(x, y')\bigr) p_t(x)dx ~\text{s.t.}~\|y-y'\|=r\Bigr\}.
    \end{aligned}
\end{equation}
For all $t\geq 0$, both $\kappa_x, \kappa_y$ are continuous in $r$ and satisfy $\lim_{r\to 0} r \kappa(r,t) = 0$ and $\lim_{r\to +\infty} r \kappa(r,t) >0$, based on Assumption \ref{assump:K}. In particular, the existence of $l_X, l_Y$ ensures that $\kappa_x, \kappa_y$ stay bounded for all  $t\geq 0$.

We now state the main convergence result.
\begin{thm}\label{thm:convbycoupling}
     Under Assumption \ref{assump:K}, suppose that $R\leq \sqrt{2\pi\beta^{-1}}\min\big\{\sqrt{m_x^{-1}},\sqrt{m_y^{-1}}\big\}$. With a constant $\gamma>0$ to be chosen, if
    \begin{equation*}
        \begin{aligned}
            L_X\leq \frac{c_1\varphi_1(R)}{2\gamma \eta }~~\text{and}~~
            L_Y\leq \frac{1}{2}c_2\gamma\eta \varphi_2(R) ,
        \end{aligned}
    \end{equation*}
 then there exist constants $0<c< \min\{c_1, \eta c_2\}, A=2\max\{\varphi_1(R)^{-1}, \gamma^{-1}\varphi_2(R)^{-1}\}$ such that
 \begin{equation}\label{eq:wassdecay}
     W_1((p_t,q_t), (p^*, q^*))\leq A\max\{1,\gamma\}e^{-ct} W_1((p_0,q_0), (p^*, q^*)).
 \end{equation}
Here, constants are set as $\varphi_1(R) = \exp(-\frac{\beta}{4}\int_0^R \tilde{\kappa}_x(s)sds)$, $\varphi_2(R) = \exp(-\frac{\beta}{4}\int_0^R \tilde{\kappa}_y(s)sds)$ with  $\tilde{\kappa}(r) := \max\{0, -\inf_{t\geq0}\kappa(r,t)\}$, and
     \begin{equation*}
        \begin{aligned}
             c_1 &= 4\beta^{-1} \Big((e-1)R^2/2 + \sqrt{\frac{8}{\beta \kappa_x}}Re^{\pi/4}+ \frac{4}{\beta \kappa_x}\Big)^{-1},\\
            c_2 &= 4\beta^{-1} \Big((e-1)R^2/2 + \sqrt{\frac{8}{\beta \kappa_y}}Re^{\pi/4}+ \frac{4}{\beta \kappa_y}\Big)^{-1},
        \end{aligned}
    \end{equation*}
    which are independent of $\eta$. 
\end{thm}
\begin{remark}
From the convergence result above, we observe that a higher temperature $\beta^{-1}$, as well as larger convexity $(\kappa_x)$ and concavity $(\kappa_y)$ parameters, can help relax the geometric conditions on $K(x,y)$ needed to guarantee convergence. Furthermore, when $\eta\gg 1$ or $\eta\ll 1$, choosing $\gamma = \eta^{-1}$ further relaxes the restrictions on $L_X, L_Y$.

It may appear from the convergence rate $\propto \min\{1, \eta\}$ that one should choose a large $\eta$ to accelerate convergence; however, the prefactor in \eqref{eq:wassdecay} would become excessively large as well, scaling with $\max\{\eta, \eta^{-1}\}$. More importantly, in order to simulate \eqref{eq:meanfieldGDA} when $\eta \gg 1$, small time step size has to be taken due to the stiffness, thus algorithmically, it is most advantageous to work in a regime away from the near quasi-static regimes (neither $\eta \ll 1$ nor $\eta \gg 1$). 
\end{remark}

\paragraph{Proof of Theorem \ref{thm:convbycoupling} } The proof is divided into several parts. 
In Part 1, we begin by subtracting the stochastic differential equations and introducing a combination of chosen concave functions to define a $L^1$ distance function. Applying It\^{o}'s formula to this distance function provides a foundation for analyzing the dynamics. Then in Part 2, we refine the analysis of the interplay between the drift and noise and identify the contractivity constant. This is achieved by constructing an appropriate concave and increasing function $f$ for the
$L^1$ distance function (Claim \ref{claim:nov5}). Using the contraction constants derived in Part 2,  we then in Part 3 select suitable parameters such as gradient Lipschitz constants for $K$ to obtain contraction. Finally, in Part 4, we conclude by comparing distance functions to establish a Wasserstein-1 distance contraction result.

\paragraph{Part 1} We start by subtracting equations in (\ref{oct291}) to get
\begin{subequations}
    \begin{align}
        dZ_t &= -\Big(\int_{\mathcal{Y}} \nabla_x K(X_t,y) q_t(y) dy -\int_{\mathcal{Y}} \nabla_x K(\tilde{X}_t,\tilde{y}) \tilde{q}_t(\tilde{y}) d\tilde{y}\Big)~dt+2\sqrt{2\beta^{-1}}\frc_1 e_t^1e_t^{1,\top} dW_t^{\frc}.\\
        dQ_t & = \eta \Big(\int_{\mathcal{X}} \nabla_y K(x,Y_t) p_t(x) dx-\int_{\mathcal{X}} \nabla_y K(\tilde{x},\tilde{Y}_t) \tilde{p}_t(\tilde{x}) d\tilde{x}\Big)~ dt+2\sqrt{2\eta\beta^{-1}}\frc_2 e_t^2e_t^{2,\top} dB_t^{\frc}.
    \end{align}
\end{subequations}
Then by It\^{o}'s formula, 
\begin{equation*}
    \begin{aligned}
d\|Z_t\| &= -e_t^1\cdot \Big(\int_{\mathcal{Y}} \nabla_x K(X_t,y) q_t(y) dy -\int_{\mathcal{Y}} \nabla_x K(\tilde{X}_t,\tilde{y}) \tilde{q}_t(\tilde{y}) d\tilde{y}\Big)~dt+2\sqrt{2\beta^{-1}}\frc_1 e_t^{1,\top} dW_t^{\frc},\\
d\|Q_t\| & = \eta e_t^2 \cdot \Big(\int_{\mathcal{X}} \nabla_y K(x,Y_t) p_t(x) dx-\int_{\mathcal{X}} \nabla_y K(\tilde{x},\tilde{Y}_t) \tilde{p}_t(\tilde{x}) d\tilde{x}\Big)~ dt+ 2 \sqrt{2\eta\beta^{-1}}\frc_2 e_t^{2,\top} dB_t^{\frc},
    \end{aligned}
\end{equation*}
since there is no It\^{o} correction due to $\d^2_{z/|z|}|z|=0$ for $z\neq 0$ and the noise coefficient function $\frc$ vanishes for $z = 0$. 
We now write $r_1(t)  = \|Z_t \|, r_2(t)= \| Q_t\|$, and choose strictly increasing concave functions $f_1, f_2 \in C^1([0,\infty))\cap  C^2((0,\infty)) $ such that $f_1(0) = f_2(0)=0$. We consider the metric
\begin{equation}
    \rho_t =  f_1(r_1(t)) +  \gamma f_2(r_2(t)).
\end{equation}
Then the It\^{o}'s formula gives that
\begin{equation}\label{nov1_1}
\begin{aligned}
    d\rho_t &= 2f_1'(r_1(t)) \sqrt{2\beta^{-1}}\frc_1 e_t^{1,\top} dW_t^{\frc} + 2 f_2'(r_2(t))\gamma \sqrt{2\eta\beta^{-1}}\frc_2e_t^{2,\top} dB_t^{\frc}\\
    &~~~ -e_t^1\cdot \Big(\int_{\mathcal{Y}} \nabla_x K(X_t,y) q_t(y) dy -\int_{\mathcal{Y}} \nabla_x K(\tilde{X}_t,\tilde{y}) \tilde{q}_t(\tilde{y}) d\tilde{y}\Big)f_1'(r_1(t)) dt\\
    &~~~+\gamma \eta e_t^2 \cdot \Big(\int_{\mathcal{X}} \nabla_y K(x,Y_t) p_t(x) dx-\int_{\mathcal{X}} \nabla_y K(\tilde{x},\tilde{Y}_t) \tilde{p}_t(\tilde{x}) d\tilde{x}\Big)f_2'(r_2(t)) dt\\
    &~~~+ 4\beta^{-1}\frc_1^2 f_1''(r_1(t)) dt + 4\gamma\beta^{-1}\eta  \frc_2^2 f_2''(r_2(t)) dt.
\end{aligned}
\end{equation}
Note that $\hat{W}_t = \int_0^t e_s^{1,\top} dW_s^{\frc}$ and $\hat{B}_t = \int_0^t e_s^{2,\top} dB_s^{\frc}$ are independent one-dimensional Brownian motions. 
\paragraph{Part 2}The next step is to analyze the deterministic part in (\ref{nov1_1}). 
As $e_t^1 =  (X_t-\tilde{X}_t)/ r_1(t)$, we get
\begin{equation}
    \begin{aligned}
       e_t^1&\cdot\Big(\int_{\mathcal{Y}} \nabla_x K(X_t,y) q_t(y) dy -\int_{\mathcal{Y}} \nabla_x K(\tilde{X}_t,\tilde{y}) \tilde{q}_t(\tilde{y}) d\tilde{y}\Big) \\
       & = \int_{\mathcal{Y}}(X_t-\tilde{X}_t)\cdot \big(\nabla_x K(X_t, y)- \nabla_x K(\tilde{X}_t, y\big) q_t(y)dy/ r_1(t)\\
       &\quad + (X_t-\tilde{X}_t)\cdot\Big(\int_{\mathcal{Y}} \nabla_x K(\tilde{X}_t,y) q_t(y) dy -\int_{\mathcal{Y}} \nabla_x K(\tilde{X}_t,\tilde{y}) \tilde{q}_t(\tilde{y}) d\tilde{y}\Big)/r_1(t)\\
       &\geq\kappa_x(r_1(t),t) \|X_t-\tilde{X}_t\|-L_Y \E[\|Y_t-\tilde{Y}_t\|].
    \end{aligned}
\end{equation}
And similarly,
\begin{equation}
    \begin{aligned}
        e_t^2\cdot\Big(\int_{\mathcal{X}} \nabla_y K(x,Y_t) p_t(x) dx& -  \int_{\mathcal{X}} \nabla_y K(\tilde{x},\tilde{Y}_t) \tilde{p}_t(\tilde{x}) d\tilde{x}\Big)\\
        &\leq -\kappa_y(r_2(t),t) \|Y_t-\tilde{Y}_t\|+L_X\E[\|X_t-\tilde{X}_t\|].
    \end{aligned}
\end{equation}
Therefore, the deterministic part in (\ref{nov1_1}) can be bounded by
\begin{equation}\label{eqn:dec2_1}
    \begin{aligned}
        &-e_t^1\cdot \Big(\int_{\mathcal{Y}} \nabla_x K(X_t,y) q_t(y) dy -\int_{\mathcal{Y}} \nabla_x K(\tilde{X}_t,\tilde{y}) \tilde{q}_t(\tilde{y}) d\tilde{y}\Big)f_1'(r_1(t)) \\
    &~~~+\gamma \eta e_t^2 \cdot \Big(\int_{\mathcal{X}} \nabla_y K(x,Y_t) p_t(x) dx-\int_{\mathcal{X}} \nabla_y K(\tilde{x},\tilde{Y}_t) \tilde{p}_t(\tilde{x}) d\tilde{x}\Big)f_2'(r_2(t)) \\
    &~~~+ 4\beta^{-1}\frc_1^2 f_1''(r_1(t))  + 4\gamma\beta^{-1}\eta  \frc_2^2 f_2''(r_2(t)) \\
    &\leq - \kappa_x(r_1(t),t) r_1(t) f_1'(r_1(t))+4\beta^{-1}\frc_1^2 f_1''(r_1(t))-\gamma \eta\kappa_y(r_2(t),t) r_2(t) f_2'(r_2(t))+4\gamma\beta^{-1}\eta  \frc_2^2 f_2''(r_2(t)) \\
    &~~~ +\gamma \eta L_X\E[\|X_t-\tilde{X}_t\|]f_2'(r_2(t))+L_Y \E[\|Y_t-\tilde{Y}_t\|]f_1'(r_1(t)) = :\Delta.
    \end{aligned}
\end{equation}
As in \cite{eberle2016reflection}, one can construct $f_1, f_2$ to nicely control the right side of (\ref{eqn:dec2_1}) for contraction. 
\begin{claim}\label{claim:nov5}
    There exist $c_1, c_2>0$ and strictly increasing concave functions $f_1, f_2$ with $f_1(0)=f_2(0)=0$ and $0\leq f_1', f_2'\leq 1$ such that for all $r>0, t\geq 0$,
    \begin{equation}\label{eqn:dec2_2}
        \begin{aligned}
            & - \kappa_x(r,t) r f_1'(r)+4\beta^{-1} f_1''(r)\leq -c_1 f_1(r),\\
            & - \kappa_y(r,t) r f_2'(r)+4\beta^{-1} f_2''(r)\leq -c_2f_2(r). 
        \end{aligned}
    \end{equation}
    If $R\leq \sqrt{2\pi\beta^{-1}}\min\big\{\sqrt{m_x^{-1}},\sqrt{m_y^{-1}}\big\}$, then we may choose 
    \begin{equation}
        \begin{aligned}
            c_1 &= 4\beta^{-1} \Big((e-1)R^2/2 + \sqrt{\frac{8}{\beta \kappa_x}}Re^{\pi/4}+ \frac{4}{\beta \kappa_x}\Big)^{-1},\\
            c_2 &= 4\beta^{-1} \Big((e-1)R^2/2 + \sqrt{\frac{8}{\beta \kappa_y}}Re^{\pi/4}+ \frac{4}{\beta \kappa_y}\Big)^{-1}.
        \end{aligned}
    \end{equation}
\end{claim}
We leave the construction and proof of the claim to Appendix \ref{sec:appendixb}.

\paragraph{Part 3} Note that $0\leq f_i'\leq 1$, and $\frc_i(Z_t, Q_t)<1$ only if $r_i(t)<\delta$. Then by Claim \ref{claim:nov5}, (\ref{eqn:dec2_1}) becomes
\begin{equation}\label{eqn:dec3_1}
    \begin{aligned}    
    \Delta&\leq -\frc_1^2 c_1 f_1(r_1(t)) -\gamma\eta\frc_2^2 c_2 f_2(r_2(t)) +\gamma \eta L_X\E[\|X_t-\tilde{X}_t\|]+L_Y \E[\|Y_t-\tilde{Y}_t\|]\\
    &\quad- \Big((1-\frc_1^2)\kappa_x(r_1(t),t) r_1(t) f_1'(r_1(t))+\gamma \eta(1-\frc_2^2)\kappa_y(r_2(t),t) r_2(t) f_2'(r_2(t))\Big) \\
    &\leq - c_1 f_1(r_1(t)) -\gamma \eta c_2 f_2(r_2(t)) + (c_1+\gamma \eta c_2)\delta+\Big(\sup_{t\geq0}\sup_{r<\delta}(r\kappa_x(r,t)^-)+\gamma \eta\sup_{t\geq0}\sup_{r<\delta}(r\kappa_y(r,t)^-)\Big)\\
    &\quad +\gamma \eta L_X\E[\|X_t-\tilde{X}_t\|]+L_Y \E[\|Y_t-\tilde{Y}_t\|],
    \end{aligned}
\end{equation}
where we use the relation $f_i(r_i(t))\leq r_i(t)\leq \delta$ when $r_i(t)\leq \delta$, to obtain the bound $-\frc_i^2 c_i f_i(r_i(t)) \leq -\frc_i^2 c_i (f_i(r_i(t))-\delta) \leq - c_i (f_i(r_i(t))-\delta)$ when $r_i(t)\leq \delta$. Taking expectations on both sides of (\ref{eqn:dec3_1}), if there exists $\varepsilon <\min\{c_1, \eta c_2\}$ satisfying
\begin{equation}\label{eqn:dec3_2}
    \gamma \eta L_X\E[\|X_t-\tilde{X}_t\|]+L_Y \E[\|Y_t-\tilde{Y}_t\|] \leq \varepsilon \E[\rho_t],
\end{equation}
then we get the inequality 
\begin{equation}
    \frac{d}{dt} \E[\rho_t] \leq -c \E[\rho_t] + m(\delta),
\end{equation}
with $c = \min\{c_1-\varepsilon, \eta c_2-\varepsilon\}$ and 
\begin{equation}
    m(\delta) = (c_1+\gamma \eta c_2)\delta+\Big(\sup_{t\geq0}\sup_{r<\delta}(r\kappa_x(r,t)^-)+\gamma \eta\sup_{t\geq0}\sup_{r<\delta}(r\kappa_y(r,t)^-)\Big).
\end{equation}
By Gr\"{o}nwall's inequality, we obtain that
\begin{equation}
    \E[\rho_t] \leq e^{-ct}\E[\rho_0] + m(\delta)(1-e^{-ct})/c. 
\end{equation}
Before moving on to Wasserstein-1 distance, we need to find the condition when (\ref{eqn:dec3_2}) is valid. By (\ref{eqn:dec3_4}) and computations in Appendix \ref{sec:appendixb}, we learn that $\varphi_i(r)/2\leq f_i'(r)\leq 1$ and $ \varphi_i(R)r/2\leq f_i(r)\leq r$. Therefore, if
\begin{equation}
    \gamma \eta L_X < \frac{c_1}{2}\varphi_1(R) \quad \text{and}\quad L_Y < \frac{\gamma \eta c_2}{2}\varphi_2(R),
\end{equation}
with $\varphi_1(R) = \exp(-\frac{\beta}{4}\int_0^R \tilde{\kappa}_x(s)sds), \varphi_2(R) = \exp(-\frac{\beta}{4}\int_0^R \tilde{\kappa}_y(s)sds)$, (\ref{eqn:dec3_2}) can hold.
\paragraph{Part 4} Finally, since $\varphi_i(R)r/2\leq f_i(r)\leq r$ for $i=1,2$, for a coupling for diffusion processes $(X_t, Y_t)$ and $(\tilde{X}_t, \tilde{Y}_t)$, we learn that the law of $(X_t, Y_t), (\tilde{X}_t,\tilde{Y}_t)$, which is $(p_t, q_t), (\tilde{p}_t, \tilde{q}_t)$ respectively, follows that
\begin{equation}
\begin{aligned}
    W_1((p_t, q_t), (\tilde{p}_t, \tilde{q}_t))&\leq A \E[d_f((X_t, Y_t), (\tilde{X}_t, \tilde{Y}_t))] = A\E[\rho_t]\\
    &\leq Ae^{-ct}\E[\rho_0]+Am(\delta)(1-e^{-ct})/c\\
    &\leq A\max\{1,\gamma\}e^{-ct} W_1((p_0, q_0), (\tilde{p}_0, \tilde{q}_0))+Am(\delta)(1-e^{-ct})/c,
    \end{aligned}
\end{equation}
with $A = 2\max\{\varphi_1(R)^{-1}, \gamma^{-1}\varphi_2(R)^{-1}\}$. Due to the assumption that both $\kappa_x, \kappa_y$ are continuous in $r$ and satisfy $\lim_{r\to 0} r \kappa(r,t) = 0$, we may pass the limit $\delta\to 0$ and  minimize over all couplings of $\mu,\nu$ to remove the error term. The convergence to the MNE is done by replacing $(\tilde{p}_t, \tilde{q}_t), (\tilde{p}_0, \tilde{q}_0)$ with $(p^*, q^*)$ to reach the conclusion. $\Box$

\section{Discussion}

In this paper, we analyzed the convergence behavior of two-timescale GDA dynamics in both finite-dimensional and infinite-dimensional settings. Leveraging tools such as hypocoercivity and coupling methods, we provided rigorous results on the interplay between the learning rate ratio and convergence properties in continuous time. Specifically, we derived quantitative estimates of convergence rates for finite-dimensional quadratic games and established Wasserstein-1 convergence guarantees for mean-field GDA dynamics under reasonable assumptions.

Our results address the importance of the learning rate ratio in determining the dynamics' behavior. Notably, for the mean-field min-max optimization problem, we addressed an open question posed by \cite{wang2024open}, providing convergence guarantees of mean-field GDA for locally nonconvex-nonconcave objective functions when high temperature is considered.

Future research could explore extending the hypocoercivity approach to analyze mean-field GDA, particularly by tackling the challenges posed by its nonlinear drift structure, which is a promising direction for further study.


\section*{Acknowledgment}

This work is supported in part by National Science Foundation via awards DMS-2309378 and IIS-2403276. 

\appendix
\section{Appendix}\label{sec:appendixa}
The purpose of this section is to provide an alternative proof and perspective of \cite[Proposition 3.1]{wang2024local} using the averaging method.
\paragraph{Convergence rate when interaction dominates}
We consider the matrix $M_{\gamma} = S+\gamma L$ with $S = \begin{bmatrix}
            Q & 0\\
            0 &  R
        \end{bmatrix}$,  $L = \begin{bmatrix}
            0 & P\\
            -P^\top &  0
\end{bmatrix}$ accompanied by a sufficiently large $\gamma>0$. \\
Consider a solution $\phi(t)\in\Rm^{n+m}$ to
\begin{equation}\label{eqn:dec3_app1}
    \dot{\phi}(t) = -M_{\gamma} \phi(t) = -S\phi - \gamma L \phi.
\end{equation}
is of the form
\begin{equation}\label{eqn:dec3_app2}
    \phi(t) = a(t)v(\gamma t) + \frac{1}{\gamma} \psi(t).
\end{equation}
Here $v\in \Rm^{n+m}$ solves 
\begin{equation}\label{eqn:skew}
            \begin{aligned}
                \dot{v}(t) = - L v, \quad \text{with}~~v(0) = v_0,
            \end{aligned}
\end{equation}
$a(t)\in C^2(\Rm)$ is a scalar function, and $\psi(t)\in\Rm^{m+n}$ is to be chosen. Note that all eigenvalues of the skew-symmetric matrix $L$ are purely imaginary, so that every element in the fundamental set of solutions to (\ref{eqn:skew}) is periodic. We thus can find a common period $T_0>0$ for the solution to (\ref{eqn:skew}). Plugging in the ansatz (\ref{eqn:dec3_app2}) to (\ref{eqn:dec3_app1}) gives that
\begin{equation}\label{eqn:jul21}
    a'(t)v(\gamma t)+ \frac{1}{\gamma} \psi'(t)= -L\psi(t) -a(t)S v(\gamma t) - \frac{1}{\gamma} S \psi(t).
\end{equation}
Now we choose an arbitrary time $t_0>0$. Multiplying (\ref{eqn:jul21}) with $v(\gamma t)$ and integrating over $[t_0/\gamma, t_0/\gamma+T)$, with $T = T_0/\gamma$, then we get for the left side as
\begin{equation}\label{eqn:lhs}
    \begin{aligned}
        \text{LHS}=&\frac{1}{T}\int_{\frac{t_0}{\gamma}}^{\frac{t_0}{\gamma}+T} a'(t) v(\gamma t)^{\top} v(\gamma t) dt + \frac{1}{\gamma T}\int_{\frac{t_0}{\gamma}}^{\frac{t_0}{\gamma}+T} \psi'(t)^{\top} v(\gamma t) dt\\
        & = \frac{1}{T_0}\int_{t_0}^{t_0+T_0} a'\big(s/\gamma\big)v(s)^{\top} v(s) ds + \frac{1}{\gamma T}\int_{\frac{t_0}{\gamma}}^{\frac{t_0}{\gamma}+T} \psi'(t)^{\top} v(\gamma t) dt.
    \end{aligned}
\end{equation}
While for the right side, we have that
\begin{equation}\label{eqn:rhs}
\begin{aligned}
   \text{RHS}=&- \frac{1}{T}\int_{\frac{t_0}{\gamma}}^{\frac{t_0}{\gamma}+T} \langle L\psi(t), v(\gamma t)\rangle dt  - \frac{1}{T}\int_{\frac{t_0}{\gamma}}^{\frac{t_0}{\gamma}+T}a(t)\langle S v(\gamma t), v(\gamma t)\rangle dt \\
   & - \frac{1}{\gamma T}\int_{\frac{t_0}{\gamma}}^{\frac{t_0}{\gamma}+T} \langle S \psi( t), v(\gamma t)\rangle dt\\
   &= - \frac{1}{T_0}\int_{t_0}^{t_0+T_0} a\big(s/\gamma\big)\langle S v(s), v(s)\rangle ds + O(1/\gamma).
   \end{aligned}
\end{equation}
The above estimates can be obtained by the following reasoning: Due to (\ref{eqn:skew}) and the skew-symmetric matrix $L$, one can find $\psi(t)\in C^1(\Rm^{2n})$ with $\|\psi(t)\|, \|\psi'(t)\|<\infty$ uniformly in $t$ such that 
\begin{equation*}
\begin{aligned}
  \Big|- \frac{1}{T}\int_{\frac{t_0}{\gamma}}^{\frac{t_0}{\gamma}+T} \langle L\psi(t), v(\gamma t)\rangle dt\Big|  &=\Big| -\frac{1}{\gamma T}  \int_{\frac{t_0}{\gamma}}^{\frac{t_0}{\gamma}+T}\psi(t)^\top  v'(\gamma t) dt \Big|\\
  & \leq \Big|-\frac{1}{\gamma T}\Big(\psi(t_0/\gamma+T)-\psi(t_0/\gamma)\Big)^\top v(t_0)\Big|+\Big|\frac{1}{\gamma T}  \int_{\frac{t_0}{\gamma}}^{\frac{t_0}{\gamma}+T}\psi'(t)^\top  v(\gamma t) dt\Big|\\
  & \leq \frac{1}{\gamma}\sup_t\|\psi'(t)\|\|v_0\|+ \frac{1}{\gamma}\sup_t \|\psi'(t)\|\|v_0\| = O(1/\gamma),
\end{aligned}
\end{equation*}
where we have used $T_0$-periodicity of $v(t)$ and
$\|v(t)\| = \|v_0\|$ for all $t\geq 0$. Similarly, as $S$ has all eigenvalues bounded, there exists a constant $C>0$ such  that
\begin{equation*}
    \Big|- \frac{1}{\gamma T}\int_{\frac{t_0}{\gamma}}^{\frac{t_0}{\gamma}+T} \langle S \psi( t), v(\gamma t)\rangle dt\Big|\leq \frac{C}{\gamma}\sup_t\|\psi(t)\|\|v_0\|  = O(1/\gamma).
\end{equation*}
In (\ref{eqn:lhs}) and (\ref{eqn:rhs}), we use the Taylor's expansion on $a, a'$ over $[t_0, t_0+T_0]$ and get that
\begin{equation}
  a\big(s/\gamma\big) = a\big(t_0/\gamma\big) + O(1/\gamma) ,
  \quad, a'\big(s/\gamma\big) = a'\big(t_0/\gamma\big) + O(1/\gamma), 
\end{equation}
Combining estimates in (\ref{eqn:lhs}) and (\ref{eqn:rhs}) together, along with
\begin{equation*}
    \Big|\frac{1}{\gamma T}\int_{\frac{t_0}{\gamma}}^{\frac{t_0}{\gamma}+T} \psi'(t)^{\top} v(\gamma t) dt\Big| \leq \frac{1}{\gamma} \sup_t\|\psi'(t)\|\|v_0\| = O(1/\gamma),
\end{equation*}
we thus get
\begin{equation}
a'\big(t_0/\gamma\big)\|v_0\|^2 = -\frac{a\big(t_0/\gamma\big)}{T_0} \int_{t_0}^{t_0+T_0} \langle S v(s), v(s)\rangle ds + O(1/\gamma).
\end{equation}
Note that for $v(t) = [u(t), w(t)]^\top, u(t)\in \Rm^n, w(t)\in \Rm^m$, we have the contraction rate as
\begin{equation}\label{eqn:dec191}
\begin{aligned}
    \mu &:= \frac{1}{T_0\|v_0\|^2} \int_{t_0}^{t_0+T_0} \langle S v(s), v(s)\rangle ds=\frac{1}{T_0\|v_0\|^2} \int_{t_0}^{t_0+T_0} (u(s)^{\top} Qu(s)+ w(s)^{\top} R w(s)) ds.
    \end{aligned}
\end{equation}
Since $t_0>0$ is chosen arbitrarily, for all $t>0$, $a(t)$ follows the dynamics 
\begin{equation}
    a'(t) = -\mu a(t) + O(1/\gamma),
\end{equation}
and thus $a(t)$ will converge to a constant exponentially with a rate $\mu$. 

If $P$ is a square matrix, i.e., $m=n$ and is of full rank, then $-L$ has distinct $2n$ complex eigenvalues $\pm i \sigma_j $ with associated unit-norm eigenvector $[\mp i u_j/\sqrt{2}, w_j/\sqrt{2}]^\top$ for $1\leq j\leq n$, and  $u_j, w_j$ come from $P = U\Sigma W^\top = \sum_{j=1}^n \sigma_j u_j w_j^\top$. By (\ref{eqn:skew}),(\ref{eqn:dec191}), and $\{e^{\pm i \sigma_j}[\mp i u_j/\sqrt{2}, w_j/\sqrt{2}]^\top, 1\leq j\leq n\}$ forming a fundamental set of solutions for $v(t)$, we have the lower bound for the contraction rate as
\begin{equation}
    \mu \geq \frac{1}{2} \min_{1\leq j\leq n}\{ u_j^\top Q u_j+  w_j^\top R w_j \}.
\end{equation}

        This coincides with the finding of \cite[Proposition 3.1]{wang2024local} which shows that the convergence rate of gradient flow with the Jacobian matrix $M_\alpha = \alpha S+L, \alpha\ll 1$ is of the form
        \begin{equation}
            \mu_{M_\alpha} = \frac{1}{2}\alpha \min_{1\leq j\leq n}\{ u_j^\top Q u_j+  v_j^\top R v_j \}+ O(\alpha^3).
        \end{equation}
        Our result obtained from the averaging method agrees with above when rescaling (\ref{eqn:dec3_app1}) by dividing with $\eta$. Moreover, based on random matrix arguments, \cite[Proposition 3.2]{wang2024local} shows that
\begin{equation}
    \E[\min_{1\leq j\leq n}\{ u_j^\top Q u_j+  v_j^\top R v_j \}]\sim \frac{\text{Tr}(S)}{n},
\end{equation}
which is approximately the average of eigenvalues of the symmetric matrix $S$. Indeed, this estimate is exact when $n=m=1$: One can find the solution to (\ref{eqn:skew}) is of the form (let $v_0= [u_0,w_0]^\top$)
\begin{equation*}
    v(t)= u_0\begin{bmatrix}
            \cos t\\
           -\sin t
        \end{bmatrix}+w_0\begin{bmatrix}
            \sin t\\
            \cos t
        \end{bmatrix}
\end{equation*}
with the periodicity $T_0=2\pi$. Then the convergence rate (\ref{eqn:dec191}) can be computed exactly, assuming $Q$ is a $2\times 2$ diagonal matrix with positive diagonal entries $\{q,r\}$:
\begin{equation*}
\begin{aligned}
    \mu&= \frac{1}{2\pi\|v_0\|^2} \int_{t_0}^{t_0+2\pi} ( qu(s)^2+  r w(s)^2) ds \\
    &= \frac{1}{2\pi\|v_0\|^2} \int_{t_0}^{t_0+2\pi}  \frac{q+r}{2} \|v_0\|^2ds = \frac{q+r}{2}.
    \end{aligned}
\end{equation*}
\section{Proof of Claim \ref{claim:nov5}}\label{sec:appendixb}
The proof is a combination of showing Lemma \ref{lem:appB1} and Lemma \ref{lem:appB2}. Similar to \cite{eberle2016reflection}, we first define following two constants (with $a, b$ defined as in (\ref{eqn:nov5_1}))
\begin{equation}
    \begin{aligned}
        &R_0 := \inf_{t\geq0}\inf \{R\geq 0: \kappa(r,t)\geq 0 ~\forall~ r\geq R\},\\
        &R_1 :=\inf_{t\geq0} \inf\{R\geq R_0: \kappa(r,t)R(R-R_0)\geq 2a/b ~\forall~ r\geq R\}.
    \end{aligned}
\end{equation}
It is straightforward to see that $R_0<R_1$.
\begin{lemma}\label{lem:appB1}
    There exist a constant $c>0$ depending on $a, b, \kappa$ and a strictly increasing concave function $f$ with $f(0)=0, 0\leq f'(r)\leq 1$ so that for all $r>0, t\geq0$,
  \begin{equation}\label{eqn:nov5_1}
        a f''(r)-b \kappa(r,t) r f'(r) \leq -cf(r).
    \end{equation}
\end{lemma}
\begin{proof}
    Consider an ansatz for $f'(r)$ is of the form
    \begin{equation}\label{eqn:dec3_4}
        f'(r) = \varphi(r) g(r),
    \end{equation}
    where $1/2\leq g(r)\leq 1$ is a monotonic decreasing continuous function with $g(0)=1$. Let $\Phi(r) = \int_0^r \varphi(s) ds$, then we have the relation 
    \begin{equation}
        \frac{1}{2}\Phi(r)\leq f(r)\leq \Phi(r).
    \end{equation}
    We take $\varphi(r) = \exp(-\frac{b}{a}\int_0^r \tilde{\kappa}(s)sds)$ with $\tilde{\kappa}(r) = \max\{0, -\inf_{t\geq0}\kappa(r,t)\}$, then the ansatz yields
    \begin{equation}
    \begin{aligned}
        a f''(r) &= a \varphi'(r) g(r) + a \varphi(r) g'(r)\\
        &= -b \tilde{\kappa}(r) r \varphi(r) g(r) + a \varphi(r) g'(r) \leq  b \kappa(r,t) r f'(r) + a \varphi(r) g'(r).
         \end{aligned}
    \end{equation}
   We first investigate  the regime $(0,R_1)$.  (\ref{eqn:nov5_1})  will automatically hold from above as long as
    \begin{equation}\label{eqn:nov5_2}
        g'\leq -\frac{c}{a} f/\varphi.
    \end{equation}
If this is true over $(0,R_1)$, then
\begin{equation}
    g(R_1)\leq 1-\frac{c}{a}\int_0^{R_1} f(s)\varphi(s)^{-1}ds\leq 1-\frac{c}{2a}\int_0^{R_1} \Phi(s)\varphi(s)^{-1}ds.
\end{equation}
Because $g\in[1/2,1]$, we need to require that
\begin{equation}\label{eqn:dec3_up}
    c\leq a \Big/ \int_0^{R_1} \Phi(s)\varphi(s)^{-1}ds.
\end{equation}
On the other hand, by choosing 
\begin{equation}
    g'(r) = -\frac{\Phi(r)}{2\varphi(r)} \Big/ \int_0^{R_1} \Phi(s)\varphi(s)^{-1}ds, \quad \text{for}~~ r<R_1,
\end{equation}
(\ref{eqn:nov5_2}) is satisfied with the constant
\begin{equation}\label{eqn:dec3_eq}
    c= \frac{a}{2} \Big/ \int_0^{R_1} \Phi(s)\varphi(s)^{-1}ds.
\end{equation}
What remains is the regime $r\geq R_1$. By the definition of $R_1$, $\varphi(r)=\varphi(R_0)$ for all $r\geq R_0$ and $g(r) = 1/2$ for all $r\geq R_1$, we have that $f'(r) = \varphi(R_0)/2$ for all $r\geq R_1$. Since $\kappa(r,t)R_1(R_1-R_0)\geq 2a/b$ for all $t\geq 0$ and $r\geq R_1$, we get
\begin{equation}
    \begin{aligned}
  a f''(r)-b \kappa(r,t) r f'(r)& = -b \kappa(r,t)\varphi(R_0) r/2\leq -\frac{a\varphi(R_0)}{R_1-R_0} \cdot \frac{r}{R_1}\\
  & \leq -\frac{a\varphi(R_0)}{R_1-R_0} \cdot \frac{\Phi(r)}{\Phi(R_1)}\leq -\frac{a}{2}\Phi(r)\Big/\int_{R_0}^{R_1} \Phi(s)\varphi(s)^{-1}ds\\
  & \leq -\Big(\frac{a}{2}\Big/\int_{0}^{R_1} \Phi(s)\varphi(s)^{-1}ds\Big) f(r) = -c f(r),
    \end{aligned}
\end{equation}
with the same choice of $c$ as in (\ref{eqn:dec3_eq}). We explain that to obtain the above inequalities, we have used $\Phi(r)/r\leq \Phi(R_1)/R_1$ by the concavity of $\Phi$, as well as 
\begin{equation*}
    \begin{aligned}
        \int_{R_0}^{R_1} \Phi(s)\varphi(s)^{-1}ds&= \int_{R_0}^{R_1} \big(\Phi(R_0)+\varphi(R_0)(s-R_0)\big)\varphi(R_0)^{-1}ds\\
        &=\Phi(R_0)\varphi(R_0)^{-1} (R_1-R_0)+(R_1-R_0)^2/2\\
        &\geq \big(\Phi(R_0)+\varphi(R_0)(R_1-R_0)\big)(R_1-R_0)\varphi(R_0)^{-1}/2\\
        &= \Phi(R_1)(R_1-R_0)\varphi(R_0)^{-1}/2.
    \end{aligned}
\end{equation*}
\end{proof}
\begin{lemma}\label{lem:appB2}
We assume that $\inf_{t\geq 0}\kappa(r,t)\geq -m$ for $r\leq R$ and $\inf_{t\geq 0}\kappa(r,t)\geq K$ for $r> R$. If for (\ref{eqn:nov5_1}), $R\leq \sqrt{\frac{a\pi}{2bm}}$, then $c$ can be chosen from
\begin{equation}\label{crange}
\begin{aligned}
  \frac{a}{2}\Big((e-1)R^2/2 + \sqrt{\frac{2a}{bK}}Re^{\pi/4}+ \frac{a}{bK} \Big)^{-1} \leq  c\leq a \Big((e-1)R^2/2 + \sqrt{\frac{2a}{bK}}Re^{\pi/4}+ \frac{a}{bK}\Big)^{-1}.
\end{aligned}
\end{equation}
\end{lemma}
\begin{proof}
    The construction of $\varphi$ in Lemma \ref{lem:appB1} tells that $\varphi(r)= \varphi(R_0)$ for all $r\geq R_0$, and 
    \begin{equation}
        \Phi(r) = \Phi(R_0) + (r-R_0) \varphi(R_0) \quad \text{for}~~r\geq R_0.
    \end{equation}
    Thus, taking similar calculations as before, we have
    \begin{equation}
        \int_0^{R_1} \Phi(s) \varphi(s)^{-1} ds = \int_0^{R_0} \Phi(s) \varphi(s)^{-1} ds + \Phi(R_0)\varphi(R_0)^{-1}(R_1-R_0) + (R_1-R_0)^2/2.
    \end{equation}
For $r\leq R_0$,
\begin{equation}
\begin{aligned}
    \Phi(r) \varphi(r)^{-1} &= \int_0^r\exp\Big(\frac{b}{a}\int_t^r s\tilde{\kappa}(s)ds\Big)dt  \leq \int_0^r\exp\Big(\frac{bm}{2a}(r^2-t^2)\Big)dt \\
    &\leq \min\Big\{\sqrt{\frac{a\pi}{2bm}}, r\Big\} \exp\Big(\frac{bm r^2}{2a}\Big).
    \end{aligned}
\end{equation}
Therefore, if one only consider the case $R_0< \sqrt{\frac{a\pi}{2bm}}$, using $e^x \leq 1+(e-1)x$ for $x\in[0,1]$, we can get that
\begin{equation}
    \int_0^{R_0} \Phi(s) \varphi(s)^{-1} ds \leq \int_0^{R_0} s \exp\Big(\frac{bms^2}{2a}\Big) ds = \frac{a}{bm}\Big(\exp\big(\frac{bm R_0^2}{2a}\big)-1\Big)\leq (e-1)R_0^2/2.
\end{equation}
By the definition of $R_0, R_1$, we have 
\begin{equation}
    R_1-R_0 \leq  \frac{1}{2}\Big(\sqrt{R_0^2+\frac{8a}{bK}}-R_0\Big)\leq \sqrt{\frac{2a}{b K}}.
\end{equation}
The estimates above, together with $R_0< \sqrt{\frac{a\pi}{2bm}}$, give us the upper bound of the integration as 
\begin{equation}
\begin{aligned}
     \int_0^{R_1} \Phi(s) \varphi(s)^{-1} ds &\leq (e-1)R_0^2/2 + R_0\sqrt{\frac{2a}{b K}}\exp\Big(\frac{bm R_0^2}{2a}\Big)+ \frac{a}{bK}\\
     &\leq (e-1)R^2/2 + \sqrt{2a}Re^{\pi/4}/\sqrt{bK}+ a/(bK), \quad \text{for}~R_0\leq R\leq \sqrt{\frac{a\pi}{2bm}}.
\end{aligned}
\end{equation}
Combined with the equations (\ref{eqn:dec3_up}) and (\ref{eqn:dec3_eq}) in Lemma \ref{lem:appB1}, we thus obtain the range of $c$ (\ref{crange}).
\end{proof}
\bibliographystyle{abbrv} 
\bibliography{references}

\begin{thebibliography}{10}

\bibitem{adolphs2019local}
L.~Adolphs, H.~Daneshmand, A.~Lucchi, and T.~Hofmann.
\newblock Local saddle point optimization: A curvature exploitation approach.
\newblock In {\em The 22nd International Conference on Artificial Intelligence
  and Statistics}, pages 486--495. PMLR, 2019.

\bibitem{andersson2000survey}
J.~Andersson.
\newblock A survey of multiobjective optimization in engineering design.
\newblock {\em Department of Mechanical Engineering, Linktjping University.
  Sweden}, 2000.

\bibitem{angiuli2022unified}
A.~Angiuli, J.-P. Fouque, and M.~Lauri{\`e}re.
\newblock {Unified reinforcement Q-learning for mean field game and control
  problems}.
\newblock {\em Mathematics of Control, Signals, and Systems}, 34(2):217--271,
  2022.

\bibitem{bailey2020finite}
J.~P. Bailey, G.~Gidel, and G.~Piliouras.
\newblock Finite regret and cycles with fixed step-size via alternating
  gradient descent-ascent.
\newblock In {\em Conference on Learning Theory}, pages 391--407. PMLR, 2020.

\bibitem{balduzzi2018mechanics}
D.~Balduzzi, S.~Racaniere, J.~Martens, J.~Foerster, K.~Tuyls, and T.~Graepel.
\newblock The mechanics of n-player differentiable games.
\newblock In {\em International Conference on Machine Learning}, pages
  354--363. PMLR, 2018.

\bibitem{berthier2024learning}
R.~Berthier, A.~Montanari, and K.~Zhou.
\newblock Learning time-scales in two-layers neural networks.
\newblock {\em Foundations of Computational Mathematics}, pages 1--84, 2024.

\bibitem{borkar2008stochastic}
V.~S. Borkar.
\newblock {\em Stochastic approximation: a dynamical systems viewpoint},
  volume~9.
\newblock Springer, 2008.

\bibitem{busoniu2008comprehensive}
L.~Busoniu, R.~Babuska, and B.~De~Schutter.
\newblock A comprehensive survey of multiagent reinforcement learning.
\newblock {\em IEEE Transactions on Systems, Man, and Cybernetics, Part C
  (Applications and Reviews)}, 38(2):156--172, 2008.

\bibitem{cai2024convergence}
Y.~Cai, S.~Mitra, X.~Wang, and A.~Wibisono.
\newblock Convergence of the min-max {L}angevin dynamics and algorithm for
  zero-sum games.
\newblock {\em arXiv preprint arXiv:2412.20471}, 2024.

\bibitem{cao2023explicit}
Y.~Cao, J.~Lu, and L.~Wang.
\newblock On explicit {$L^2$}-convergence rate estimate for underdamped
  {L}angevin dynamics.
\newblock {\em Archive for Rational Mechanics and Analysis}, 247(5):90, 2023.

\bibitem{chang2020distributed}
T.-H. Chang, M.~Hong, H.-T. Wai, X.~Zhang, and S.~Lu.
\newblock Distributed learning in the nonconvex world: From batch data to
  streaming and beyond.
\newblock {\em IEEE Signal Processing Magazine}, 37(3):26--38, 2020.

\bibitem{cheng2018sharp}
X.~Cheng, N.~S. Chatterji, Y.~Abbasi-Yadkori, P.~L. Bartlett, and M.~I. Jordan.
\newblock Sharp convergence rates for {L}angevin dynamics in the nonconvex
  setting.
\newblock {\em arXiv preprint arXiv:1805.01648}, 2018.

\bibitem{cheng2018underdamped}
X.~Cheng, N.~S. Chatterji, P.~L. Bartlett, and M.~I. Jordan.
\newblock Underdamped {L}angevin {MCMC}: {A} non-asymptotic analysis.
\newblock In {\em Conference on learning theory}, pages 300--323. PMLR, 2018.

\bibitem{cherukuri2017saddle}
A.~Cherukuri, B.~Gharesifard, and J.~Cortes.
\newblock Saddle-point dynamics: conditions for asymptotic stability of saddle
  points.
\newblock {\em SIAM Journal on Control and Optimization}, 55(1):486--511, 2017.

\bibitem{conger2024coupled}
L.~Conger, F.~Hoffmann, E.~Mazumdar, and L.~J. Ratliff.
\newblock Coupled {W}asserstein gradient flows for min-max and cooperative
  games.
\newblock {\em arXiv preprint arXiv:2411.07403}, 2024.

\bibitem{dalal2020tale}
G.~Dalal, B.~Szorenyi, and G.~Thoppe.
\newblock A tale of two-timescale reinforcement learning with the tightest
  finite-time bound.
\newblock In {\em Proceedings of the AAAI Conference on Artificial
  Intelligence}, volume~34, pages 3701--3708, 2020.

\bibitem{daskalakis2020independent}
C.~Daskalakis, D.~J. Foster, and N.~Golowich.
\newblock Independent policy gradient methods for competitive reinforcement
  learning.
\newblock {\em Advances in neural information processing systems},
  33:5527--5540, 2020.

\bibitem{daskalakis2018training}
C.~Daskalakis, A.~Ilyas, V.~Syrgkanis, and H.~Zeng.
\newblock Training {GAN}s with optimism.
\newblock In {\em International Conference on Learning Representations}, 2018.

\bibitem{ding2020natural}
D.~Ding, K.~Zhang, T.~Basar, and M.~Jovanovic.
\newblock Natural policy gradient primal-dual method for constrained markov
  decision processes.
\newblock {\em Advances in Neural Information Processing Systems},
  33:8378--8390, 2020.

\bibitem{doan2022convergence}
T.~Doan.
\newblock Convergence rates of two-time-scale gradient descent-ascent dynamics
  for solving nonconvex min-max problems.
\newblock In {\em Learning for Dynamics and Control Conference}, pages
  192--206. PMLR, 2022.

\bibitem{dolbeault2015hypocoercivity}
J.~Dolbeault, C.~Mouhot, and C.~Schmeiser.
\newblock Hypocoercivity for linear kinetic equations conserving mass.
\newblock {\em Transactions of the American Mathematical Society},
  367(6):3807--3828, 2015.

\bibitem{domingo2020mean}
C.~Domingo-Enrich, S.~Jelassi, A.~Mensch, G.~Rotskoff, and J.~Bruna.
\newblock A mean-field analysis of two-player zero-sum games.
\newblock {\em Advances in neural information processing systems},
  33:20215--20226, 2020.

\bibitem{durmus2017nonasymptotic}
A.~Durmus and {\'E}.~Moulines.
\newblock Nonasymptotic convergence analysis for the unadjusted {L}angevin
  algorithm.
\newblock {\em The Annals of Applied Probability}, 27(3):1551, 2017.

\bibitem{eberle2016reflection}
A.~Eberle.
\newblock Reflection couplings and contraction rates for diffusions.
\newblock {\em Probability theory and related fields}, 166:851--886, 2016.

\bibitem{eberle2019couplings}
A.~Eberle, A.~Guillin, and R.~Zimmer.
\newblock Couplings and quantitative contraction rates for {L}angevin dynamics.
\newblock {\em The Annals of Probability}, 47(4):1982, 2019.

\bibitem{gambier2007multi}
A.~Gambier and E.~Badreddin.
\newblock Multi-objective optimal control: An overview.
\newblock In {\em 2007 IEEE international conference on control applications},
  pages 170--175. IEEE, 2007.

\bibitem{glicksberg1952further}
I.~L. Glicksberg.
\newblock A further generalization of the kakutani fixed point theorem, with
  application to {N}ash equilibrium points.
\newblock {\em Proceedings of the American Mathematical Society},
  3(1):170--174, 1952.

\bibitem{goodfellow2014generative}
I.~Goodfellow, J.~Pouget-Abadie, M.~Mirza, B.~Xu, D.~Warde-Farley, S.~Ozair,
  A.~Courville, and Y.~Bengio.
\newblock Generative adversarial nets.
\newblock {\em Advances in neural information processing systems}, 27, 2014.

\bibitem{grothaus2014hypocoercivity}
M.~Grothaus and P.~Stilgenbauer.
\newblock Hypocoercivity for {K}olmogorov backward evolution equations and
  applications.
\newblock {\em Journal of Functional Analysis}, 267(10):3515--3556, 2014.

\bibitem{heusel2017gans}
M.~Heusel, H.~Ramsauer, T.~Unterthiner, B.~Nessler, and S.~Hochreiter.
\newblock {GAN}s trained by a two time-scale update rule converge to a local
  {N}ash equilibrium.
\newblock {\em Advances in neural information processing systems}, 30, 2017.

\bibitem{hommes2012multiple}
C.~H. Hommes and M.~I. Ochea.
\newblock Multiple equilibria and limit cycles in evolutionary games with logit
  dynamics.
\newblock {\em Games and Economic Behavior}, 74(1):434--441, 2012.

\bibitem{hsieh2019finding}
Y.-P. Hsieh, C.~Liu, and V.~Cevher.
\newblock Finding mixed {N}ash equilibria of generative adversarial networks.
\newblock In {\em International Conference on Machine Learning}, pages
  2810--2819. PMLR, 2019.

\bibitem{kimsymmetric}
J.~Kim, K.~Yamamoto, K.~Oko, Z.~Yang, and T.~Suzuki.
\newblock Symmetric mean-field {L}angevin dynamics for distributional minimax
  problems.
\newblock In {\em The Twelfth International Conference on Learning
  Representations}.

\bibitem{konda2004convergence}
V.~R. Konda and J.~N. Tsitsiklis.
\newblock Convergence rate of linear two-time-scale stochastic approximation.
\newblock {\em Ann. Appl. Probab.}, 14(1):796--819, 2004.

\bibitem{korpelevich1976extragradient}
G.~M. Korpelevich.
\newblock The extragradient method for finding saddle points and other
  problems.
\newblock {\em Matecon}, 12:747--756, 1976.

\bibitem{lan2020communication}
G.~Lan, S.~Lee, and Y.~Zhou.
\newblock Communication-efficient algorithms for decentralized and stochastic
  optimization.
\newblock {\em Mathematical Programming}, 180(1):237--284, 2020.

\bibitem{lascu2024fisher}
R.-A. Lascu, M.~B. Majka, and {\L}.~Szpruch.
\newblock A {F}isher-{R}ao gradient flow for entropic mean-field min-max games.
\newblock {\em Transactions on Machine Learning Research}, pages ISSN
  2835--8856, 2024a.

\bibitem{letcher2019differentiable}
A.~Letcher, D.~Balduzzi, S.~Racaniere, J.~Martens, J.~Foerster, K.~Tuyls, and
  T.~Graepel.
\newblock Differentiable game mechanics.
\newblock {\em Journal of Machine Learning Research}, 20(84):1--40, 2019.

\bibitem{liang2019interaction}
T.~Liang and J.~Stokes.
\newblock Interaction matters: A note on non-asymptotic local convergence of
  generative adversarial networks.
\newblock In {\em The 22nd International Conference on Artificial Intelligence
  and Statistics}, pages 907--915. PMLR, 2019.

\bibitem{lin2020projection}
T.~Lin, C.~Fan, N.~Ho, M.~Cuturi, and M.~Jordan.
\newblock Projection robust wasserstein distance and {R}iemannian optimization.
\newblock {\em Advances in neural information processing systems},
  33:9383--9397, 2020.

\bibitem{lu2023two}
Y.~Lu.
\newblock Two-scale gradient descent ascent dynamics finds mixed {N}ash
  equilibria of continuous games: A mean-field perspective.
\newblock In {\em International Conference on Machine Learning}, pages
  22790--22811. PMLR, 2023.

\bibitem{ma2022provably}
C.~Ma and L.~Ying.
\newblock Provably convergent quasistatic dynamics for mean-field two-player
  zero-sum games.
\newblock {\em International Conference on Learning Representations}, 2022.

\bibitem{mertikopoulos2018cycles}
P.~Mertikopoulos, C.~Papadimitriou, and G.~Piliouras.
\newblock Cycles in adversarial regularized learning.
\newblock In {\em Proceedings of the twenty-ninth annual ACM-SIAM symposium on
  discrete algorithms}, pages 2703--2717. SIAM, 2018.

\bibitem{mou2022improved}
W.~Mou, N.~Flammarion, M.~J. Wainwright, and P.~L. Bartlett.
\newblock Improved bounds for discretization of {L}angevin diffusions:
  {N}ear-optimal rates without convexity.
\newblock {\em Bernoulli}, 28(3):1577--1601, 2022.

\bibitem{nash1951non}
J.~F. {N}ash.
\newblock Non-cooperative games.
\newblock {\em Annals of Mathematics}, 54(2), 1951.

\bibitem{nedic2009subgradient}
A.~Nedi{\'c} and A.~Ozdaglar.
\newblock Subgradient methods for saddle-point problems.
\newblock {\em Journal of optimization theory and applications}, 142:205--228,
  2009.

\bibitem{nemirovski2004prox}
A.~Nemirovski.
\newblock Prox-method with rate of convergence {O}(1/t) for variational
  inequalities with {L}ipschitz continuous monotone operators and smooth
  convex-concave saddle point problems.
\newblock {\em SIAM Journal on Optimization}, 15(1):229--251, 2004.

\bibitem{roussel2018spectral}
J.~Roussel and G.~Stoltz.
\newblock Spectral methods for {L}angevin dynamics and associated error
  estimates.
\newblock {\em ESAIM: Mathematical Modelling and Numerical Analysis},
  52(3):1051--1083, 2018.

\bibitem{schuh2024global}
K.~Schuh.
\newblock Global contractivity for {L}angevin dynamics with
  distribution-dependent forces and uniform in time propagation of chaos.
\newblock In {\em Annales de l'Institut Henri Poincare (B) Probabilites et
  statistiques}, volume~60, pages 753--789. Institut Henri Poincar{\'e}, 2024.

\bibitem{tapia2007applications}
M.~G.~C. Tapia and C.~A.~C. Coello.
\newblock Applications of multi-objective evolutionary algorithms in economics
  and finance: A survey.
\newblock In {\em 2007 IEEE congress on evolutionary computation}, pages
  532--539. IEEE, 2007.

\bibitem{villani2009hypocoercivity}
C.~Villani.
\newblock {\em Hypocoercivity}, volume 202.
\newblock American Mathematical Society, 2009.

\bibitem{wang2022exponentially}
G.~Wang and L.~Chizat.
\newblock An exponentially converging particle method for the mixed {N}ash
  equilibrium of continuous games.
\newblock {\em arXiv preprint arXiv:2211.01280}, 2022.

\bibitem{wang2024local}
G.~Wang and L.~Chizat.
\newblock Local convergence of gradient methods for min-max games: partial
  curvature generically suffices.
\newblock {\em Advances in Neural Information Processing Systems}, 36, 2024.

\bibitem{wang2024open}
G.~Wang and L.~Chizat.
\newblock Open problem: Convergence of single-timescale mean-field {L}angevin
  descent-ascent for two-player zero-sum games.
\newblock In {\em The Thirty Seventh Annual Conference on Learning Theory},
  pages 5345--5350. PMLR, 2024.

\bibitem{yang2020global}
J.~Yang, N.~Kiyavash, and N.~He.
\newblock Global convergence and variance reduction for a class of
  nonconvex-nonconcave minimax problems.
\newblock {\em Advances in Neural Information Processing Systems},
  33:1153--1165, 2020.

\end{thebibliography}

\end{document}